\documentclass[reqno,11pt]{amsart}
\usepackage{amsmath,amssymb,amsthm,amscd,amsfonts,amsxtra,mathrsfs,latexsym}
\usepackage{indentfirst}
\usepackage[top=3.5cm, bottom=4.0cm, right=2.5cm, left=2.5cm]{geometry}

\usepackage[title]{appendix}
\usepackage{bbold}
\usepackage{bbm}
\usepackage{dsfont}

\usepackage[hyphens]{url}
\usepackage{mathtools}
\usepackage[msc-links]{amsrefs}

\usepackage{hyperref}
\usepackage{bookmark}

\allowdisplaybreaks
\mathtoolsset{showonlyrefs=true}
\AtBeginDocument{\def\MR#1{}}

\newcounter{mparcnt}

\usepackage{fancyhdr}
\usepackage{esint}
\usepackage{enumerate}
\usepackage{xcolor}

\usepackage{pictexwd,dcpic}
\usepackage{graphicx}
\usepackage{caption}
\usepackage{graphicx}
\usepackage{slashed}
\usepackage{epsfig,here}
\usepackage{subfigure,here}

\usepackage[hyphens]{url}
\usepackage{hyperref}
\usepackage{bookmark}
\usepackage{amsmath,thmtools,mathtools}
\allowdisplaybreaks 
\mathtoolsset{showonlyrefs=true}

\numberwithin{equation}{section}
\title{\textbf{Sub-Finslerian interpolation inequalities}}
\author{Haowei Lin}
\address{Graduate School of Science, the University of Osaka, Osaka 560-0043, Japan}
\email{u466248a@ecs.osaka-u.ac.jp}
 \date{\today}
\theoremstyle{plain}
\newtheorem{theorem}{Theorem}[section]
\newtheorem{lemma}[theorem]{Lemma}
\newtheorem{proposition}[theorem]{Proposition}
\newtheorem{corollary}[theorem]{Corollary}

\theoremstyle{definition}
\newtheorem{definition}[theorem]{Definition}
\newtheorem{example}[theorem]{Example}

\newtheorem{remark}[theorem]{Remark}

\newcommand \ep {\epsilon}

\newcommand \supp {\text{Supp}}
\newcommand \ann {\text{Ann}}
\newcommand \mj {\mathcal{J}}

\newcommand \m {\mathfrak{m}}
\newcommand \Cut[1] {\mathrm{Cut}(#1)}

\begin{document}
\bibliographystyle{plain}
\nocite{*}

\maketitle

\setcounter{page}{0}
\maketitle
\thispagestyle{empty}
\begin{abstract}

In this paper, 
we prove that forward ideal sub-Finslerian manifolds support interpolation inequalities for optimal transport, 
extending the results of Barilari and Rizzi \cite{BR} from the sub-Riemannian to the sub-Finslerian setting. 
A key role is played by the introduction of sub-Finslerian Jacobi fields and the establishment of optimal transport 
theory on sub-Finslerian manifolds. By combining this transport framework with sub-Finslerian Jacobian estimates, we characterize the 
generalized distortion coefficients. As an application, we deduce several fundamental geometric inequalities, 
including the Brunn-Minkowski and Borell-Brascamp-Lieb inequalities. Finally, for the case of the Randers 
sub-Finslerian Heisenberg group, whose metric is defined by a sub-Riemannian metric perturbed by a drift term,
we explicitly show that it satisfies the measure contraction property.

\end{abstract}
\pagenumbering{Roman}
\setcounter{page}{1}
\tableofcontents
\setcounter{page}{1}
\pagenumbering{arabic}

\section{Introduction}
With the rapid development of metric measure geometry, nonholonomic geometry has attracted considerable attention 
in recent years. Nonholonomic geometry refers to geometric structures where the metric is only defined on a totally 
non-integrable distribution. If the metric is given by a positive definite quadratic form, it is called sub-Riemannian geometry; 
sub-Finslerian or sub-Lorentzian geometries can be defined in an analogous way. We refer to the comprehensive 
textbook \cite{Agr} for a systematic foundation and to \cite{Rf} for connections with optimal transport theory 
in the sub-Riemannian setting.

One of the most elusive and intriguing features of nonholonomic geometry is the general failure of the classical 
curvature-dimension (CD) conditions. These conditions were originally introduced independently by Lott-Villani 
\cite{LV} and Sturm \cites{St1,St2} for metric measure spaces, but they typically break down in nonholonomic 
settings (we refer to \cites{Jui,RS23,Borza,MR} for related discussions on this failure in various contexts). 
Motivated by the close connection between the Brunn-Minkowski 
inequality and curvature-dimension conditions, we use the former to obtain preliminary insights into the failure 
of the CD condition. To this end, let us introduce the model coefficient $\tau^{(t)}_{K,N}(\theta)$. 
For every $K \in \mathbb{R}$, $N \in (1, \infty)$ and $t\in (0,1)$, it is defined as:
\begin{equation}\label{eq:model}
    \tau^{(t)}_{K,N}(\theta) := \begin{cases}
        t^{\frac{1}{N}} \left( \frac{\sin(t\theta\sqrt{K/(N-1)})}{\sin(\theta\sqrt{K/(N-1)})} \right)^{\frac{N-1}{N}} & \text{if } (N-1)\pi^2 > K\theta^2 > 0, \\
        t & \text{if } K = 0, \\
        t^{\frac{1}{N}} \left( \frac{\sinh(t\theta\sqrt{-K/(N-1)})}{\sinh(\theta\sqrt{-K/(N-1)})} \right)^{\frac{N-1}{N}} & \text{if } K < 0.
    \end{cases}
\end{equation}

We say that a Finslerian manifold $M$ equipped with a smooth measure $\m$ satisfies the Brunn-Minkowski inequality $\text{BM}(K,N)$ if, for any nonempty Borel subsets $A, B \subset M$ and $t \in (0,1)$, we have
\begin{equation}\label{bm0}
    \m(Z_t(A,B))^{\frac{1}{N}} \ge \tau^{(1-t)}_{K,N}(\Theta(A,B))\cdot \m(A)^{\frac{1}{N}} + \tau^{(t)}_{K,N}(\Theta(A,B))\cdot \m(B)^{\frac{1}{N}},
\end{equation}
where $d$ denotes the Finslerian distance on $M$, $Z_t(A,B)$ denotes the set of $t$-intermediate points of geodesics connecting $A$ and $B$, and
\begin{equation}
    \Theta(A, B) := \begin{cases}
        \inf\limits_{x\in A, y\in B} d(x,y) & \text{if } K \ge 0, \\
        \sup\limits_{x\in A, y\in B} d(x,y) & \text{if } K < 0.
    \end{cases}
\end{equation}
This property has been established by Sturm \cite{St1} for metric measure spaces satisfying CD$(K,N)$, and then Ohta \cite{Ohta} showed that CD$(K,N)$ holds for Finslerian manifolds such that the weighted Ricci curvature $\mathrm{Ric}_N$ is bounded below by $K$.

The situation changes drastically, however, in the sub-Finslerian setting, 
where many effective arguments from the classical holonomic case (i.e., Riemannian or Finslerian geometry) break down entirely,
 and $\mathrm{BM}(K,N)$ 
fails for any choice of $K$ and $N$ (see \cite[Theorem 4.26]{MR}). This failure and the associated technical 
difficulties stem from two main obstructions. The first obstruction lies in the model coefficients
$\tau_{K,N}^{(t)}$ appearing in inequality \eqref{bm0}. These functions are constructed from space forms, which are 
questionable comparison objects in our setting since neither a Chern connection nor a satisfactory curvature theory 
is available. Moreover, there is a fundamental discrepancy between the topological dimension and 
the geodesic dimension, the latter describing the infinitesimal rate of volume distortion; 
in the sub-Finslerian setting, geodesic dimension is strictly larger than the topological dimension (see Section 5.1 for a detailed discussion). The second major complication arises from the regularity of the sub-Finslerian distance function. Due to the existence of abnormal geodesics, the distance function lacks the necessary regularity, which poses a severe technical hurdle for constructing optimal transport maps.

Inspired by the work of Barilari and Rizzi \cite{BR} on sub-Riemannian interpolation inequalities, 
we restrict our focus to forward ideal sub-Finslerian manifolds---namely, forward complete manifolds that do not 
admit non-trivial abnormal minimizing geodesics. Furthermore, by replacing the model coefficients with a generalized 
definition of distortion coefficients, we successfully establish a sub-Finslerian version of the interpolation 
inequality.

\begin{theorem}[Interpolation Inequality]\label{interpolation}
    Let $(M,\mathcal{D},F)$ be a forward ideal sub-Finslerian manifold and let $\m$ be a fixed smooth measure on $M$. Take $\mu_0,\mu_1\in \mathcal{P}^{ac}_c(M)$. Let $T$ be the unique optimal transport map from $\mu_0$ to $\mu_1$, and let $T_t$ denote the $t$-intermediate optimal transport map. Let $\mu_t=\rho_t \m=(T_t)_\sharp \mu_0$, for $t\in[0,1]$, be the unique Wasserstein geodesic joining $\mu_0$ and $\mu_1$. Then for all $t\in[0,1]$, it holds that 
    \begin{equation}
        \frac{1}{\rho_t(T_t(x))^{1/n}}\ge \frac{\beta^{>}_t(x,T(x))^{1/n}}{\rho_0(x)^{1/n}}+\frac{\beta^{<}_t(x,T(x))^{1/n}}{\rho_1(T(x))^{1/n}}\quad \text{ for }\mu_0\text{-a.e. } x\in M.
    \end{equation}
\end{theorem}
Here, $\beta^{>}_t(x,y)$ and $\beta^{<}_t(x,y)$ denote the generalized backward and forward distortion coefficients, respectively. Their rigorous definitions, which involve the limits of measure ratios along sub-Finslerian geodesic flows, are deferred to Section 5.1.

Our generalization mainly is twofold. First, as an essential step, we generalize the work of Figalli and Rifford \cite{FR} on optimal transport from the sub-Riemannian to the sub-Finslerian setting. Second, we develop the notions of Jacobi fields and distortion coefficients for sub-Finslerian manifolds to conclude our main result, Theorem \ref{interpolation}. Because of the asymmetry of the sub-Finslerian distance $d_{SF}(x,y)$, this generalization is non-trivial, and we must carefully track the order of the pair $(x,y)$. Several necessary ingredients are also developed in this article, such as the analysis of conjugate points, which is well-understood in sub-Riemannian geometry but requires careful extension to the sub-Finslerian case.

The organization of the paper is as follows. 
In Section 2, we introduce the basic notions of sub-Finslerian geometry. We define the sub-Finslerian distance, naturally introduce sub-Finslerian geodesics, and discuss related geometric concepts such as the sub-Finslerian exponential map. We also address the regularity of the sub-Finslerian distance function. 
Section 3 is devoted to the study of Jacobi fields on sub-Finslerian manifolds, where we establish a Jacobian estimate (Theorem \ref{jest}) that plays an essential role in proving our main interpolation inequalities. 
In Section 4, we develop the sub-Finslerian optimal transport theory and prove a Brenier--McCann type theorem (Theorem \ref{ot}). 
In Section 5, we provide a proof of Theorem \ref{interpolation} and derive the Borell--Brascamp--Lieb and Brunn--Minkowski inequalities as direct corollaries. 
In Section 6, we analyze a specific example, namely a Randers sub-Finslerian Heisenberg group. By explicitly computing its distortion coefficients and establishing its measure contraction property, we provide a concrete, illustrative example that empirically corroborates the recent findings of Borza, Magnabosco, Rossi, and Tashiro \cite{BMRT1} on general sub-Finslerian Heisenberg groups. 
Finally, in Section 7, we briefly discuss the manifestation of two synthetic curvature conditions in the sub-Finslerian setting: Barilari--Mondino--Rizzi's unified theory \cite{BMR} and Milman's Quasi Curvature-Dimension (QCD) condition \cite{Mil}.

\section{Sub-Finslerian Geometry}
Sub-Finslerian manifolds are generalizations of Finslerian manifolds, where metric is defined only on a non-integrable subbundle of the tangent bundle. 
In this section, we recall some basic settings and definitions of sub-Finslerian geometry. We refer to \cite{MR} and \cite{YAD} for the sub-Finslerian context, and to \cite{Agr} for a comprehensive introduction to sub-Riemannian geometry. We also refer to \cite{BCS} and \cite{Oh0} for background on Finsler geometry.

\subsection{Sub-Finslerian structure}
\begin{definition}
    Let $M$ be an $n$-dimensional smooth manifold, and let $k\in \mathbb{N}$. We call $(M,\sigma,F)$ a \textit{sub-Finslerian manifold} if the pair $(\sigma,F)$ satisfies the following conditions:
    \begin{itemize}
        \item[(1)] $\sigma:M\times \mathbb{R}^k\to TM$ is a smooth morphism of vector bundles over the identity map of $M$;       
        \item[(2)] $F:M\times \mathbb{R}^k\to [0,\infty)$ is a continuous function and for any fixed $p\in M$, $F_p := F(p,\cdot)$ is a Minkowski norm on $\mathbb{R}^k$; namely,\begin{itemize}
            \item[(i)] $F_p:\mathbb{R}^k\setminus \{0\}\to [0,\infty)$ is smooth;
            \item[(ii)] $F_p(\lambda u)=\lambda F_p(u)$ for $\lambda>0,u\in \mathbb{R}^k$;
            \item[(iii)]  $\frac{1}{2}F_p^2$ is strongly convex;
        \end{itemize} 
        \item[(3)] The distribution $\mathcal{D}=\{\sigma(s)\mid s\in\Gamma(M\times\mathbb{R}^k)\}$ is bracket-generating.
    \end{itemize} 
\end{definition}

\begin{remark}
    This definition generalizes sub-Riemannian geometry and is usually referred to as a free sub-Finslerian structure. A completely general definition can be found in \cite[Definition 4.2]{YAD}. The key distinction is that their norm is defined on a general vector bundle, whereas ours is defined on a trivial bundle $M \times \mathbb{R}^k$. However, any completely general sub-Finslerian structure is equivalent to a free sub-Finslerian structure \cite[Theorem 4.5]{YAD}. Therefore, for the sake of simplicity in exposition, we adopt the present definition.
    An alternative way to generalize sub-Riemannian structures was introduced by Magnabosco and Rossi \cite[Definition 2.10]{MR}. In their setting, the norm is required not to vary on the fibers. This constitutes the main difference from our definition. Fortunately, this distinction does not affect the applicability of their results in our setting; we will provide the necessary clarifications when their work is invoked.
\end{remark}

Note that the Minkowski norm on the distribution is not required to be reversible, namely, $F(-v)$ is not necessarily equal to $F(v)$ with $v\neq0$, which fundamentally distinguishes it from the sub-Riemannian setting. By identifying the dual bundle $(M\times \mathbb{R}^k)^*$ with $M\times \mathbb{R}^k$ via the standard Euclidean inner product, the dual norm $F^*$ on $M\times \mathbb{R}^k$ is defined as 
\begin{equation}
    F^*(\xi):=F(\ell^{-1}(\xi)),
\end{equation}
where $\ell:M\times (\mathbb{R}^k\setminus\{0\})\to \Big(M\times (\mathbb{R}^k\setminus\{0\})\Big)^*=M\times \Big(\mathbb{R}^n\setminus \{0\}\Big)$ is the \textit{Legendre transformation} of $F$, explicitly defined via the fundamental tensor $g$ as
\begin{equation}
    \ell(u)(v) = g_u(u,v) = \left. \frac{1}{2}\frac{\partial^2}{\partial s \partial t}\right|_{s=t=0} F^2(u+su+tv).
\end{equation}

Let $\{e_i\}_{i=1}^k$ be a standard basis of $\mathbb{R}^k$. We define a \textit{generating frame} $\{X_i\}_{i=1}^k$ on $TM$ by 
\begin{equation}\label{frame}
    X_i(p)=\sigma(p,e_i).
\end{equation}
The dimension of the distribution $\mathcal{D}_p=\text{span}\{X_i(p)\}$ at $p\in M$ is called the \textit{rank} of $p$. Here, we do not require the rank of the distribution to be constant over the entire manifold. On the distribution, the induced sub-Finslerian norm is defined as follows.

\begin{definition}[Sub-Finslerian norm]
    Let $(M,\sigma,F)$ be a sub-Finslerian manifold. For $v\in \mathcal{D}_p\subset T_pM$, we define the \textit{sub-Finslerian norm} of $v$ (also denoted by $F$) as
    \begin{equation}\label{eq:2.3}
        F(v):=\min\limits_u\{F_p(u)\mid u\in \mathbb{R}^k, \sigma(p,u) = v\}.
    \end{equation}
\end{definition}
Since $\sigma$ is linear with respect to $u$ and $F_p$ is strongly convex with respect to $u$, the minimizer attaining \eqref{eq:2.3} exists and is unique. In what follows, to emphasize the role of the distribution, we will denote the sub-Finslerian manifold as $(M,\mathcal{D},F)$, omitting the process of using the morphism $\sigma$ to construct the distribution $\mathcal{D}$; $F$ should be understood as defined on the distribution in the way described above.
\begin{definition}
    A curve $\gamma:[0,T]\to M$ is called an \textit{admissible curve} if it is absolutely continuous and $\dot{\gamma}(t)\in \mathcal{D}_{\gamma(t)}$ for almost every $t\in [0,T]$. In this case, there exists $u=(u_1,...,u_k)\in L^2([0,T];\mathbb{R}^k)$ such that 
    \begin{equation}
        \dot{\gamma}(t)=\sigma(\gamma(t), u(t))=\sum_{i=1}^k u_i(t)X_i(\gamma(t)),\quad \text{and} \quad F(\dot{\gamma}(t))=F(u(t)) \text{ for a.e. }t\in[0,T].
    \end{equation}
    The function $u$ is called the \textit{minimal control}. The \textit{length} of $\gamma$ is defined as 
    \begin{equation}
        L(\gamma):=\int^T_0 F(\dot{\gamma}(t))dt.
    \end{equation}
    We say that $\gamma$ is parameterized by arc length if $F(\dot{\gamma}(t))=1$ for almost every $t\in[0,T]$.
\end{definition}

The \textit{sub-Finslerian distance} is defined as
\begin{equation}
    d_{SF}(p,q):=\inf\{L(\gamma)\mid \gamma:[0,1]\to M\text{ is an admissible curve from $p$ to $q$} \}.
\end{equation}
An admissible curve that minimizes the length among all admissible curves with the same endpoints is called a \textit{minimizing geodesic}, i.e., $L(\gamma)=d_{SF}(\gamma(0),\gamma(1))$. An analogous local existence result for minimizing geodesics can be obtained as in the sub-Riemannian setting (see \cite[Theorem 5.5]{YAD}).

We introduce the sub-Finslerian \textit{forward ball} by
$B^+_r(x)=\{p\in M \mid d_{SF}(x,p)< r\}$, and the \textit{backward ball} is defined by $B^-_r(x)=\{p\in M \mid d_{SF}(p,x)< r\}$. A sequence $(x_n)_{n\in\mathbb{N}}$ in $M$ is called \textit{forward Cauchy} if for every $\varepsilon>0$ there exists $N\in\mathbb{N}$ such that $d_{SF}(x_n,x_m)<\varepsilon$ for all $m\ge n\ge N$. A sub-Finslerian manifold $M$ is called \textit{forward complete} if every forward Cauchy sequence converges. A Hopf--Rinow type theorem holds in the sub-Finslerian setting: forward completeness is equivalent to the compactness of $\overline{B^+_r(x)}$ for every $x\in M$ and $r>0$, and, under either of these equivalent conditions, every pair of points in $M$ can be joined by a minimizing geodesic. We refer the reader to \cite[Theorem 5.8]{YAD} and \cite[Theorem 11]{AK}.

To characterize these minimizing geodesics, instead of the length functional, it is often more convenient to consider the energy functional
\begin{equation}
    J(\gamma) := \frac{1}{2}\int^1_0 F(\dot{\gamma}(t))^2 dt.
\end{equation}
It is a standard fact that an admissible curve minimizes the energy functional if and only if it minimizes the length functional and is parameterized with constant speed.

Equipped with this distance, $(M,d_{SF})$ becomes an asymmetric metric space, and $d_{SF}:M\times M\to \mathbb{R}$ is finite and continuous (see \cite[Theorem 4.12]{YAD}). Moreover, since the distance is induced from a Minkowski norm on the trivial vector bundle $M\times \mathbb{R}^k$, it is locally equivalent to any reference sub-Riemannian metric $d_{SR}$ induced by an arbitrary smooth inner product on the same distribution $\mathcal{D}$; namely, for any compact subset $K \subset M$, there exists a constant $C>0$ such that for all $p, q \in K$,
\begin{equation}\label{sfsr}
    \frac{1}{C}d_{SR}(p,q)\le d_{SF}(p,q)\le Cd_{SR}(p,q).
\end{equation}
Since the metric is generally irreversible, we introduce the reversed metric structure. 
Set $\bar{F}(v):=F(-v)$ and $\bar{d}_{SF}(x,y):=d_{SF}(y,x)$. 
Throughout the paper, objects associated with the reversed metric structure, 
such as the Hamiltonian and the exponential map, will be denoted by either an overbar or a left arrow above the symbol; the specific notation will be clarified when needed.

\subsection{Exponential Map}

Let $\gamma:[0,1]\to M$ be an admissible curve. Letting $u = (u_1, \dots, u_k) \in L^2([0,1];\mathbb{R}^k)$ be its minimal control, the curve satisfies 
\begin{equation}
    \dot{\gamma}(t) = \sum_{i=1}^k u_i(t)X_i(\gamma(t)) \quad \text{for a.e. } t\in[0,1].
\end{equation}
In this setting, the energy functional $J(\gamma)$ can be reformulated directly on the control space $L^2([0,1];\mathbb{R}^k)$ as 
\begin{equation}
    J(u) = \frac{1}{2}\int^1_0 F(u(t))^2 dt.
\end{equation}
By a slight abuse of notation, we use the same symbol $J$ for both functionals. In classical Riemannian geometry, geodesics are derived via standard calculus of variations. However, in the sub-Finslerian setting, the non-holonomic constraint restricts admissible variations, rendering classical variational methods insufficient. Therefore, the search for minimizing geodesics is formulated as an optimal control problem, where the Pontryagin Maximum Principle (PMP) provides the precise first-order necessary conditions.

\begin{theorem}[see {\cite[Theorem 6.4]{YAD}}]
    Let $\gamma:[0,1]\to M$ be a minimizing geodesic parameterized by constant speed, and let $u$ be the corresponding minimal control. Denote by $P_{0,t}$ the flow of the non-autonomous vector field $X_u(t, \cdot) := \sum_{i=1}^k u_i(t)X_i(\cdot)$. 
    Then there exists a covector $\lambda_0 \in T^*_{\gamma(0)}M$ such that 
    \[\lambda(t) := (P^{-1}_{0,t})^*\lambda_0, \quad \lambda(t) \in T^*_{\gamma(t)}M\] 
    satisfies one of the following condition:
    \begin{itemize}
        \item[(N)] $u(t) = \ell^{-1}(\bar{\lambda}(t))$, where $\bar{\lambda}(t) = \sigma^*(\lambda(t))$;
        \item[(A)] $0 \equiv \langle \lambda(t), X_i(\gamma(t)) \rangle$ for all $i=1,\dots,k$, where $\langle \cdot,\cdot \rangle$ denotes the canonical paring.
    \end{itemize}
    Moreover, in the second case, one has $\lambda_0 \neq 0$. 
\end{theorem}

A covector curve $\lambda(t)$ satisfying the condition (N) is called a normal extremal, 
and one satisfying the condition (A) is called an abnormal extremal. A geodesic is called normal 
(resp. abnormal) if it is the projection of a normal (resp. abnormal) extremal. To 
characterize normal extremals, we introduce the sub-Finslerian Hamiltonian $H: T^*M \to \mathbb{R}$:

\begin{equation}\label{hm}
    H(\lambda) = \sup_{u \in \mathbb{R}^k} \left\{ \langle \lambda, \sigma(p, u) \rangle - \frac{1}{2}F_p^2(u) \right\}, \quad \text{where } p = \pi(\lambda).
\end{equation}
Unlike sub-Riemannian geometry, the sub-Finslerian Hamiltonian loses regularity on the 
\emph{annihilator} of the distribution, which is defined as 
\begin{equation}
    \text{Ann}(\mathcal{D}) = \{\lambda \in T^*M \mid \langle \lambda, w \rangle = 0 \ \forall w \in \mathcal{D}_{\pi(\lambda)}\}.
\end{equation}

\begin{proposition}\label{prop2.5}
    Let $(M,\mathcal{D},F)$ be a sub-Finslerian manifold. Then 
    \begin{equation}\label{hmt0}
        H(\lambda) = \frac{1}{2}(F^* \circ \sigma^*(\lambda))^2 \quad \forall \lambda \in T^*M.
    \end{equation}
    Moreover, $H^{-1}(0) = \text{Ann}(\mathcal{D})$, and $H$ is smooth on $T^*M \setminus \text{Ann}(\mathcal{D})$ and continuous on 
    $T^*M$. Additionally, in local coordinates, 
    \begin{equation}\label{f25}
        H(\lambda) = \frac{1}{2} \frac{\partial^2 H}{\partial p^i\partial p^j}(\lambda)p^i p^j \quad \forall \lambda = (x^1,\dots,x^n, p^1,\dots,p^n) \in T^*M \setminus \text{Ann}(\mathcal{D}).
    \end{equation}
\end{proposition}
\begin{proof}
    Let $p = \pi(\lambda)$. By definition, $\langle \lambda, \sigma(p, u) \rangle = \langle \sigma^*(\lambda), u \rangle$. For $\lambda \notin \text{Ann}(\mathcal{D})$, meaning $\sigma^*(\lambda) \neq 0$, differentiating \eqref{hm} with respect to $u$ shows that the supremum is attained for $u = \ell_p^{-1}(\sigma^*(\lambda))$, where $\ell_p: \mathbb{R}^k \setminus \{0\} \to (\mathbb{R}^k)^* \setminus \{0\}$ is the Legendre transformation of $F_p$. Substituting this back yields $H(\lambda) = \frac{1}{2}(F^* \circ \sigma^*(\lambda))^2$. If $\lambda \in \text{Ann}(\mathcal{D})$, then $\sigma^*(\lambda) = 0$, the supremum is trivially attained at $u=0$, and $H(\lambda) = 0$, which is consistent with the formula.
    
    The smoothness of $H(\lambda)$ outside the annihilator follows directly from \eqref{hmt0}, 
    since $F^*$ is smooth outside the zero section. Moreover, $\lambda \in \text{Ann}(\mathcal{D})$ 
    if and only if $H(\lambda) = 0$. For $\lambda \notin \text{Ann}(\mathcal{D})$, $H$ is positively 
    2-homogeneous with respect to the fiber coordinates; namely, for any $c > 0$, 
    $H(c\lambda) = c^2 H(\lambda)$. Applying Euler's homogeneous function theorem
     gives the Hessian representation in \eqref{f25}.
\end{proof}

For $\lambda \in T^*M \setminus \text{Ann}(\mathcal{D})$, the Hamiltonian vector field $\vec{H}$ is defined via the canonical symplectic form $\omega$ on $T^*M$ by 
\begin{equation}
    d_{\lambda}H = \omega(\cdot, \vec{H}(\lambda)).
\end{equation}
In local Darboux coordinates $(x^i, p_i)$, $\vec{H}$ can be expressed as 
\begin{equation}
    \vec{H}(\lambda) = \sum_{i=1}^n \left( \frac{\partial H}{\partial p_i}(\lambda)\frac{\partial }{\partial x^i} - \frac{\partial H}{\partial x^i}(\lambda)\frac{\partial }{\partial p_i} \right).
\end{equation}

\begin{proposition}[see {\cite[Theorem 7.2]{YAD}}]
    Let $(M,\mathcal{D},F)$ be a sub-Finslerian manifold. If a curve $\lambda(t)$ is a normal extremal and $\lambda(0) \in \text{Ann}(\mathcal{D})$, then $\lambda(t) \equiv \lambda(0)$. If $\lambda(0) \in T^*M \setminus \text{Ann}(\mathcal{D})$, then 
    \begin{equation}
        \dot{\lambda}(t) = \vec{H}(\lambda(t)).
    \end{equation}
\end{proposition}

This Hamiltonian flow allows us to define the exponential map on the cotangent bundle.

\begin{definition}
    Let $(M,\mathcal{D},F)$ be a forward complete smooth sub-Finslerian manifold, and let $q \in M$. The \textit{sub-Finslerian exponential map} $\exp_q: T_q^*M \to M$ is defined by 
    \begin{equation}
        \exp_q(\lambda) := \begin{cases}
            \pi \circ e^{\vec{H}}(\lambda) \quad &\text{if } \lambda \in T_q^*M \setminus \text{Ann}(\mathcal{D}_q);\\
            q \quad &\text{if } \lambda \in \text{Ann}(\mathcal{D}_q),
        \end{cases}
    \end{equation}
    where $\pi$ denotes the canonical projection, and $e^{\vec{H}}$ denotes the flow of $\vec{H}$ evaluated at time $t=1$.
\end{definition}

Another way to describe extremals is via the endpoint map and Lagrange multipliers.

\begin{definition}[Endpoint map]
 Let $(M,\mathcal{D},F)$ be a sub-Finslerian manifold and let $q \in M$. Let $\mathcal{U}_q \subset L^2([0,1]; \mathbb{R}^k)$ be the open set of all controls for which the corresponding trajectory $\gamma_u$ starting at $q$ is defined on $[0,1]$. We define the \textit{endpoint map} based at $q$ as 
    \begin{equation}
        E_q: \mathcal{U}_q \to M, \quad E_q(u) = \gamma_u(1).
    \end{equation} 
\end{definition}

The differentials of the endpoint map can be computed via the Volterra series expansion. For a detailed proof in the sub-Riemannian setting---which is adapted straightforwardly to the sub-Finslerian case---we refer to \cite[Chapter 8]{Agr}.

\begin{proposition}[Variation formulas for endpoint maps]\label{var_endpoint}
    Let $\gamma_u: [0,1] \to M$ be a normal geodesic, joining $x$ to $y$. Let $P_{t,1}$ denote the flow generated by the non-autonomous vector field $X_{u(t)} = \sum_{i=1}^k u_i(t)X_i$. Then for any $v \in \ker D_u E_x \subset L^2([0,1]; \mathbb{R}^k)$, we have 
    \begin{equation}
        \begin{aligned}
            D_u E_x(v) &= \int^1_0 (P_{t,1})_* X_{v(t)}(\gamma_u(t)) \, dt, \\
            D^2_u E_x(v,v) &= \iint\limits_{0 \le s \le t \le 1} \left[ (P_{s,1})_* X_{v(s)}(\gamma_u(s)), (P_{t,1})_* X_{v(t)}(\gamma_u(t)) \right](y) \, ds dt,
        \end{aligned}
    \end{equation}
    where $X_{v(t)} = \sum_{i=1}^k v_i(t)X_i$.
\end{proposition}

The Lagrange multiplier rule provides an equivalent characterization of extremals.

\begin{proposition}\label{lagmul}
    Let $(M,\mathcal{D},F)$ be a sub-Finslerian manifold and let $\gamma_u$ be a minimizing geodesic
     connecting $x$ to $y$ associated with the control $u$. Then there exists a non-zero 
     covector $\lambda \in T_y^*M$ such that for any $v \in L^2([0,1]; \mathbb{R}^k)$, 
     one of the following conditions is satisfied:
    \begin{itemize}
        \item[(N)] $\langle \lambda, D_u E_x(v) \rangle = D_u J(v)$;
        \item[(A)] $\langle \lambda, D_u E_x(v) \rangle = 0$.
    \end{itemize}
    Moreover, the conditions (N) and (A) provide characterizations of 
    normal and abnormal extremals, respectively.
\end{proposition}

We end this subsection with the following important result, which is a key ingredient for discussing the regularity of the sub-Finslerian distance.
\begin{proposition}\label{prop1}
    Let $x, y \in M$ be two distinct points and assume that there is a function 
    $\phi: M \to \mathbb{R}$ differentiable at $x$ such that 
    \begin{equation}
        d^2_{SF}(x,y) = \phi(x) \quad \text{and} \quad d^2_{SF}(z,y) \ge \phi(z) \quad
         \forall z \in M.
    \end{equation}
    Then, the geodesic joining $x$ and $y$ is unique, has a normal lift to $T^*M$ and is given by 
    $\gamma : [0,1] \to M$; $\gamma(t) = \exp_x(-\frac{1}{2}t d_x\phi)$. In particular, 
    $y = \exp_x(-\frac{1}{2}d_x\phi)$.
\end{proposition}
\noindent The proof follows standard arguments from the sub-Riemannian context; 
see, for instance, \cite[Lemma 2.15]{Rf}.

\subsection{Regularity of the sub-Finslerian distance}
In this subsection, we recall the regularity of the sub-Finslerian distance function; 
the results presented here have been established in \cites{MR,YAD}. We also refer to 
\cite[Chapter 11]{Agr} for the corresponding discussion in the sub-Riemannian setting, 
which inspires the related research in sub-Finslerian geometry.

\begin{definition}[Conjugate point]
    Let $(M,\mathcal{D},F)$ be a sub-Finslerian manifold and let $\gamma:[0,1]\to M$ be a normal 
    geodesic with $\gamma(t)=\exp_p(t\lambda)$.
    We say that $q=\exp_p(\bar{t}\lambda)$ is a \textit{conjugate point} to $p$ along $\gamma$ if $\bar{t}\lambda$
    is a critical point for $\exp_p$.
\end{definition}

The properties of conjugate points in sub-Finslerian geometry are elusive. For instance, unlike the Riemannian or Finslerian settings, conjugate points may accumulate. This peculiar fact arises from the presence of abnormal segments. However, if we restrict our discussion to normal geodesics containing no abnormal segments, the behavior of conjugate points resembles that in Riemannian or Finsler geometry.

\begin{definition}
    We say that a normal geodesic $\gamma:[0,1]\to M$ contains \textit{no abnormal segments} if for every 
    $0\le s_1 < s_2\le 1$, the restriction $\gamma|_{[s_1,s_2]}$ is not abnormal.
\end{definition}

\begin{theorem}\label{conjugatepoint}
    Let $\gamma:[0,1]\to M$ be a minimizing geodesic that does not contain abnormal segments. Then 
    $\gamma(s)$ is not conjugate to $\gamma(s')$ for every $s,s'\in[0,1]$ with $|s-s'|<1$.
\end{theorem}

To the best of our knowledge, this result has not been explicitly established in the sub-Finslerian literature, although it is a well-known fact in the sub-Riemannian setting (see \cite[Appendix A]{BR} and \cite[Chapter 8]{Agr} for detailed discussions). 
However, the standard sub-Riemannian techniques apply straightforwardly to the sub-Finslerian framework. 
The core idea relies on the fact that the existence of conjugate points implies that the Hessian of the energy 
functional $J$ admits a negative eigenvalue. Consequently, one can construct a strictly shorter variation, contradicting
 the minimality of the geodesic. 
The rigorous proof of Theorem \ref{conjugatepoint} is deferred to Appendix A, 
as it heavily relies on the discussion concerning conjugate points presented therein.

Now that we have established that normal geodesics admitting no abnormal segments 
behave regularly concerning conjugate points, we can apply these properties 
to study the regularity of the sub-Finslerian distance function. Let us introduce the following set of smooth points, on which, as the name suggests, 
the squared distance function is smooth.
\begin{definition}
    Let $(M,\mathcal{D},F)$ be a sub-Finslerian manifold and fix $p\in M.$ The \textit{set of smooth points from $p$} is the subset $\Sigma_p\subset M$ of points $q\in M$ such that there exists a unique length-minimizer $\gamma$ joining $p$ and $q$ which admits no abnormal extremals and such that $q$ is not conjugate to $p$ along $\gamma$. Define the \textit{cut locus of $p\in M$} as $\text{Cut}(p):=M\setminus \Sigma_p$. Finally, the \textit{cut locus of $M$} is the set 
    \begin{equation}
        \text{Cut}(M):=\{(p,q)\in M\times M \mid q\in \text{Cut}(p)\}\subset M\times M.
    \end{equation}
\end{definition}

Carrying out the argument adapted from sub-Riemannian geometry (see \cite[Chapter 11]{Agr}), 
the absence of conjugate points guarantees the maximal rank of the endpoint map, which via the Implicit 
Function Theorem yields the following smoothness result. 

\begin{theorem}
    The set of smooth points $\Sigma_p$ is an open subset of $M$, and $d^2_{SF}(p,\cdot)$
     is $C^{\infty}$ on $\Sigma_p$ for any fixed point $p\in M.$
\end{theorem}

As we have seen, the sub-Finslerian distance function behaves quite similarly to the sub-Riemannian setting provided that abnormal geodesics are excluded from consideration. Therefore, we introduce the following fundamental assumption for the rest of the paper.

\begin{definition}
    Let $(M,\mathcal{D},F)$ be a sub-Finslerian manifold. We call it \textit{forward ideal} if it is forward complete and admits no nontrivial abnormal minimizing geodesics. 
\end{definition}
This terminology was first introduced by Rifford in the sub-Riemannian setting (in which forward completeness is equivalent to completeness; we therefore simply call it \textit{ideal}) (see \cite{Rf0,Rf}). The ideal structure has been shown to be generic in sub-Riemannian geometry; see the related discussion in \cite{CJT}. In the sub-Finslerian setting, although this property has not been extensively investigated, the sub-Finslerian Heisenberg group provides a nontrivial example of a forward ideal sub-Finslerian manifold.
On forward ideal sub-Finslerian manifolds, although the global smoothness of the distance 
function still fails to hold (for local smoothness, we refer to \cite[Section 4.2]{MR}), 
we can prove that the squared distance function is locally semiconcave outside the diagonal. This regularity is precisely what is needed to establish the Brenier--McCann theorem for optimal transport, which we will address in Section 4.

\begin{theorem}[cf. {\cite[Theorem 3.14]{Rf}}]\label{lip}
    Let $(M,\mathcal{D},F)$ be a forward ideal sub-Finslerian manifold. Then the squared sub-Finslerian distance $d^2_{SF}$ is locally 
    semiconcave on $M\times M\setminus \Delta$, where $\Delta$ denotes the diagonal. In particular, $d_{SF}$ is 
    locally Lipschitz on $M\times M\setminus \Delta$.
\end{theorem}
\begin{proof}
    It suffices to construct a $C^2$ upper support function for the squared distance function locally around any pair of distinct points. Fix $(x,y)\in M\times M\setminus \Delta$ and let $\gamma_u$ be a minimizing normal geodesic with control $u$ from $x$ to $y$. Because $M$ is forward ideal, $\gamma_u$ admits no abnormal segments; hence, the differential of the endpoint map $D_u E_x : L^2([0,1];\mathbb{R}^k) \to T_y M$ has full rank. We can therefore select $n$ control variations $v^1,...,v^n\in L^2([0,1];\mathbb{R}^k)$ such that the restricted linear map $\alpha=(\alpha_1,...,\alpha_n)\mapsto \sum_{i=1}^n \alpha_iD_{u}E_x(v^i)$ is an isomorphism onto $T_y M$.

    By working in local coordinates, we locally identify small neighborhoods of $x$ and $y$ with open subsets of $\mathbb{R}^n$. Under this identification, we can define the operator:
    \begin{align}
        \mathcal{F}&: \mathbb{R}^n \times \mathbb{R}^n \to \mathbb{R}^n\times \mathbb{R}^n\\
        &(z,\alpha)\mapsto \left(z,E_z\left(u+\sum_{i=1}^n \alpha_iv^i\right)\right). \nonumber
    \end{align}
    This map $\mathcal{F}$ is well-defined and of class $C^2$ in a neighborhood of $(x,0) \in \mathbb{R}^n \times \mathbb{R}^n$, with $\mathcal{F}(x,0)=(x,y)$. Since the differential of the endpoint map is an isomorphism on the span of $\{v^i\}$, the differential $d\mathcal{F}$ is invertible at $(x,0)$. By the Implicit Function Theorem, there exists an open neighborhood $\mathcal{B} \subset \mathbb{R}^n \times \mathbb{R}^n$ centered at $(x,y)$ and a $C^2$ inverse map $\mathcal{G}: \mathcal{B} \to \mathbb{R}^n \times \mathbb{R}^n$ such that 
    \begin{equation}
        \mathcal{F}(\mathcal{G}(z,w))=(z,w), \quad \forall (z,w)\in \mathcal{B}.
    \end{equation}
    Let $\alpha^{-1}(z,w)=\pi_2\circ\mathcal{G}(z,w)$, where $\pi_2$ denotes the projection onto the second $\mathbb{R}^n$ component, and define the perturbed control $\tilde{u}_{z,w} := u + \sum_{i=1}^n (\alpha^{-1}(z,w))_i v^i$. By definition, $\tilde{u}_{z,w}$ is a valid control from $z$ to $w$. Therefore, defining the function
    \begin{equation}
        \phi_{x,y}(z,w) := \int^1_0 F^2\left(\dot{\gamma}_{z,\tilde{u}_{z,w}}\right)dt,
    \end{equation}
    where $\gamma_{z, \tilde{u}_{z,w}}$ denotes the trajectory 
    starting at $z$ with control $\tilde{u}_{z,w}$, we have $d^2_{SF}(z,w)\le \phi_{x,y}(z,w)$ for 
    every $(z,w)\in \mathcal{B}$, with equality $d^2_{SF}(x,y)=\phi_{x,y}(x,y)$ holding at the center.
     By a standard compactness argument (cf. \cite[Theorem 3.14]{Rf} and \cite[Corollary 5.4]{YAD}), 
     the set of minimizing geodesics connecting any pair of points in a given compact subset 
     $\mathcal{B} \subset M \times M \setminus \Delta$ is compact with respect to the uniform topology. Therefore
     the $C^2$ norms of the family of functions $\phi_{x,y}(\cdot,\cdot)$ are uniformly bounded locally with respect to $x,y\in M\times M\setminus \Delta$. Hence, $d^2_{SF}$ is locally semiconcave.
\end{proof}
\subsection{Some examples}
As an end of this section, we introduce some examples.
\begin{example}
    The Heisenberg group $\mathbb{H}$ is $\mathbb{R}^3$ equipped with the group law
\begin{equation}
    (x,y,z)\cdot (x',y',z')=\left(x+x',y+y',z+z'+\frac{1}{2}(xy'-yx')\right).
\end{equation}
Let the vector bundle morphism $\sigma:\mathbb{R}^3\times \mathbb{R}^2\to T\mathbb{H}$ be given by 
\begin{equation}
    \sigma(x,y,z,u_1,u_2)=\left(x,y,z,\,u_1,\,u_2,\,\frac{x}{2}u_2-\frac{y}{2}u_1\right),
\end{equation} 
which induces the distribution $\mathcal{D}:=\{X:=\partial_x-\frac{y}{2}\partial_z,\,Y:=\partial_y+\frac{x}{2}\partial_z\}$. Equipped with an arbitrary Minkowski norm $F$ on $\mathbb{R}^2$, the triple $(\mathbb{H},\sigma,F)$ constitutes a sub-Finslerian manifold, which we call the sub-Finslerian Heisenberg group. 
Being one of the most important and simplest nonholonomic models, the sub-Finslerian Heisenberg group has been the subject of a wealth of results. We refer to the series of works by Borza, Magnabosco, Rossi, and Tashiro~\cites{BMR,BMRT1,BMRT24,MR}. We will also discuss its properties further in Section~6.
\end{example}
\begin{example}    
    Let $M$ be an $n$-dimensional manifold, and let $\mathcal{D}\subset TM$ be a bracket-generating sub-bundle of rank $k<n$. Take $F_{g,\beta}:\mathcal{D}\to [0,+\infty)$ defined by
    \begin{equation}\label{rnorm}
        F_{g,\beta}(v)=g(v,v)^{1/2}+\beta(v),
    \end{equation}
    where $g$ is a sub-Riemannian metric on $\mathcal{D}$, and $\beta\in \Gamma(\mathcal{D}^*)$ is a smooth horizontal $1$-form on $M$ with $\|\beta\|_g<1$;
    here $\mathcal{D}^*$ denotes the dual bundle of $\mathcal{D}$, given by $\mathcal{D}^*\simeq T^*M\slash \ann(\mathcal{D})$. We call $(M,\mathcal{D},F_{g,\beta})$ a Randers sub-Finslerian manifold. We refer to the recent work of Alabdulsada \cite{Ab} on this, wherein he calls it a sub-Randers manifold.
\end{example}

\section{Sub-Finslerian Jacobi equation}

Let $M$ be a smooth $n$-dimensional manifold and let $(x^i,p_i)$ be the canonical local Darboux coordinates on the cotangent bundle $T^*M$, induced by a local coordinate chart $(U,x^i)$ on $M$. Let $\lambda(t)=e^{t\vec{H}}(\lambda_0)=(x(t),p(t))$ be an integral curve of the Hamiltonian flow, which satisfies Hamiltonian equations: 
\begin{equation}
    \begin{cases}
        \dot{x}^i=\frac{\partial H}{\partial p_i}(x(t),p(t));\\
        \dot{p}_i=-\frac{\partial H}{\partial x^i}(x(t),p(t)).
    \end{cases}
\end{equation}
For any smooth vector field $\xi$ along $\lambda(t)$, its Lie derivative in the direction of the Hamiltonian vector field $\vec{H}$ is defined by 
\begin{equation}
    \dot{\xi}(t):=\frac{d }{d \ep}\bigg|_{\ep=0}\left(e^{-\ep\vec{H}}_*\xi(t+\ep)\right).
\end{equation}
We say that a vector field $\mj$ along $\lambda(t)$ is a \textit{Jacobi field} if $\dot{\mj}(t)=0$. Given an initial condition $\mj(0)\in T_{\lambda_0}(T^*M)$, the corresponding unique Jacobi field is given by $\mj(t)=e^{t\vec{H}}_*\mj(0)$.

The following result is a generalization to the sub-Finslerian setting of the classical fact that, in Riemannian geometry, Jacobi fields arise from variations of geodesics.

\begin{proposition}
    A vector field $\mj(t)$ along a Hamiltonian trajectory $\lambda(t)$ is a Jacobi field if and only if it is the variational vector field of a one-parameter family of normal extremals.
\end{proposition}
\begin{proof}
    Consider a one-parameter family of normal extremals,
    \begin{equation}
        \lambda(t,s)=e^{t\vec{H}}(\eta(s)),
    \end{equation}
    where $s \mapsto \eta(s)$ is a smooth curve in $T^*M$ satisfying $\eta(0) = \lambda_0$. The variation field is given by $\frac{\partial}{\partial s}\big|_{s=0}\lambda(t,s)=e^{t\vec{H}}_*\left(\frac{d}{ds}\big|_{s=0}\eta(s)\right)$. Letting $\mj(0) = \frac{d}{ds}\big|_{s=0}\eta(s)$, we have $\mj(t) = e^{t\vec{H}}_*\mj(0)$. Therefore, 
    \begin{equation}
        \frac{d}{d\ep}\bigg|_{\ep=0}\left(e^{-\ep\vec{H}}_*\mj(t+\ep)\right) = \frac{d}{d\ep}\bigg|_{\ep=0}\left(e^{-\ep\vec{H}}_*e^{(t+\ep)\vec{H}}_*\mj(0)\right) = \frac{d}{d\ep}\bigg|_{\ep=0}\left(e^{t\vec{H}}_*\mj(0)\right) = 0,
    \end{equation}
    which implies that the variation field is a Jacobi field.

    Conversely, let $\mj(t)=e^{t\vec{H}}_*\mj(0)$ be a Jacobi field, with $\mj(0)\in T_{\lambda_0}(T^*M)$. Let $s \mapsto \eta(s)$ be any smooth curve in $T^*M$ passing through $\lambda_0$ at $s=0$ with tangent vector $\frac{d}{ds}\big|_{s=0}\eta(s) = \mj(0)$. Then $\lambda(t,s)=e^{t\vec{H}}(\eta(s))$ is a one-parameter family of Hamiltonian trajectories, and its variational vector field is precisely
    \begin{equation}
        \frac{\partial}{\partial s}\bigg|_{s=0}\lambda(t,s)=e^{t\vec{H}}_*\mj(0)=\mj(t).
    \end{equation} 
\end{proof}

On the cotangent bundle $T^*M$, the vertical subspace is canonically defined as
\begin{equation}
    \mathcal{V}_\lambda = \ker(d\pi_\lambda) = T_{\lambda}\big(T^*_{\pi(\lambda)}M\big) = \operatorname{span}\{\partial_{p_i}\}.
\end{equation}
The horizontal subspace $\mathcal{H}_\lambda$ is defined as a complementary subspace transverse to the vertical subspace $\mathcal{V}_\lambda$. Once a local coordinate chart $(x^i, p_i)$ is fixed, $\mathcal{H}_\lambda$ is typically chosen as 
$\mathcal{H}_\lambda = \operatorname{span}\{\partial_{x^i}\}.$
Jacobi fields have a deep connection with conjugate points of $\gamma(t)$ as stated in the following lemma.
\begin{lemma}\label{cjlm}
    Let $\gamma(t)=\exp_x(t\lambda_0)$, $\lambda_0\in T^*_{x}M$, be a normal geodesic on $M$, and let $\lambda(t) = e^{t\vec{H}}(\lambda_0)$ be its corresponding normal extremal. Then for any $s\in(0,1]$, $\gamma(s)$ is not conjugate to $\gamma(0)$ along $\gamma(t)$ if and only if
    \begin{equation}
        e^{s\vec{H}}_*\mathcal{V}_{\lambda_0}\cap \mathcal{V}_{\lambda_s}=\{0\},
    \end{equation}
    where $\lambda_i:=\lambda(i),i=0,s.$
    Moreover, let $\mathcal{H}_{\lambda_i} \subset T_{\lambda_i}(T^*M)$ be a subspace 
    transverse to $\mathcal{V}_{\lambda_i}$ for $i \in \{0, s\}$. Then for any 
    pair $(J_0,J_s)\in \mathcal{H}_{\lambda_0}\times \mathcal{H}_{\lambda_s}$, there exists 
    a unique Jacobi field $\mj(t)$ along $\lambda(t)$, $t\in [0,1]$, such that the 
    projection of $\mj(i)$ onto $\mathcal{H}_{\lambda_i}$ along $\mathcal{V}_{\lambda_i}$ is 
    equal to $J_i$ for $i \in \{0, s\}$.
\end{lemma}
\begin{proof}
    By definition, the point $\gamma(s)$ is not conjugate to $\gamma(0)$ if $s\lambda_0$ is not a critical point of the exponential map $\exp_x$. This is equivalent to saying that the differential of exponential map satisfies
    \begin{equation}
        d\exp|_{\lambda_s} v=d\pi|_{\lambda_s}\cdot e^{s\vec{H}}_*|_{\lambda_0}v\neq 0,\quad \forall v\in T_{\lambda_0}\big(T^*_{x}M\big)\setminus\{0\}=\mathcal{V}_{\lambda_0}\setminus\{0\}.
    \end{equation}
     This exactly means that the pushed-forward vertical space intersects the kernel of $d\pi|_{\lambda_s}$
    trivially, i.e.,
    \begin{equation}
        e^{s\vec{H}}_*\mathcal{V}_{\lambda_0}\cap \mathcal{V}_{\lambda_s}=\{0\}.
    \end{equation}
    
    For the second statement, an arbitrary Jacobi field is determined by its initial condition $\mj(0)$. Since $J_0\in\mathcal{H}_{\lambda_0}$, we must have $\mj(0) = J_0 + V_0$ for some $V_0\in\mathcal{V}_{\lambda_0}$. By the definition of the sub-Finslerian Jacobi field, $\mj(s)=e^{s\vec{H}}_*J_0+e^{s\vec{H}}_*V_0$. We want the horizontal projection of $\mj(s)$ to be $J_s$, which means we require
\begin{equation}
    \pi_{\mathcal{H}}(e^{s\vec{H}}_*V_0) = \pi_{\mathcal{H}}(J_s - e^{s\vec{H}}_*J_0),
\end{equation}
where $\pi_{\mathcal{H}}:T^*M\to \mathcal{H}$ is the projection onto horizontal subbundle.
Since $e^{s\vec{H}}_*\mathcal{V}_{\lambda_0}\cap \mathcal{V}_{\lambda_s}=\{0\}$, the horizontal projection restricts to an isomorphism on $e^{s\vec{H}}_*\mathcal{V}_{\lambda_0}$. Therefore, for the given target horizontal data $J_s$, the vertical component $V_0$ at $\lambda_0$ is uniquely determined. This completes the proof.
\end{proof}

\subsection{Riccati-type equation}
Let $\omega$ be the canonical symplectic form on $T^*M$ and let $\gamma(t)$ be a normal geodesic with the corresponding 
normal extremal $\lambda(t)$. 
Let $E_1(t), \dots, E_n(t), F_1(t), \dots, F_n(t) \in T_{\lambda(t)}(T^*M)$ be a Darboux moving frame along $\lambda(t)$ such that
\begin{equation}
    \omega(E_i, E_j) = \omega(F_i, F_j) = 0, \quad \omega(E_i, F_j) = \delta_{ij},
\end{equation}
with $\pi_*E_i=0,X_i:=\pi_* F_i\in TM, \forall i,j=1,\dots n$. Here $X_i$ denotes the corresponding moving frame along $\gamma(t).$ In this case, $E_1, \dots, E_n$ span the vertical subbundle $\mathcal{V}_{\lambda(t)}$, and the vectors $F_1, \dots, F_n$ span a horizontal subspace $\mathcal{H}_{\lambda(t)}$ transverse to the vertical subspace. We also say that the frame $\{E_i, F_i\}$ is a \emph{Darboux lift} of $X_i$.

Consider an $n$-dimensional subspace generated by a set of linearly independent Jacobi fields $\mj_1(t), 
\dots, \mj_n(t)$ along $\lambda(t)$, that is,
\begin{equation}
    \mathcal{L}_t = \operatorname{span}\{\mj_1(t), \dots, \mj_n(t)\}.
\end{equation}
We identify it with a smooth family of $2n \times n$ matrices, called \emph{Jacobian matrices},
\begin{equation}
    \mathbf{J}(t) = \begin{pmatrix} M(t)\\ N(t) \end{pmatrix}, \quad t\in[0,1],
\end{equation}
such that with respect to the Darboux frame, we have
\begin{equation}
    \mj_i(t) = \sum_{j=1}^n M_{ji}(t)E_j(t) + N_{ji}(t)F_j(t) \quad\forall i=1,...,n.
\end{equation}

By the condition $\dot{\mj}_i(t) = 0$, the derivatives of $M(t)$ and $N(t)$ are completely determined by the derivatives of the Darboux frame vectors $\dot{E}_i$ and $\dot{F}_i$ along the Hamiltonian flow. Expanding these frame derivatives using the symplectic form (we refer to \cite{BR} for details, whose derivations apply verbatim to our setting), we obtain the Jacobi equation:
\begin{equation}\label{jeq}
    \frac{d}{dt}\mathbf{J}(t) = \begin{pmatrix}
        -A(t) & -R(t)\\
        B(t) & A^T(t)\\
    \end{pmatrix}\mathbf{J}(t),
\end{equation}
where $R(t)$ and $B(t)$ are symmetric matrices and $B(t) \ge 0$. More precisely, extracting the coefficients via the symplectic form $\omega$ yields:
\begin{align}\label{abr}
    A_{ij} &= -\omega_{\lambda(t)}(F_i, \dot{E}_j); \\
    R_{ij} &= -\omega_{\lambda(t)}(F_i, \dot{F}_j); \\
    B_{ij} &= \omega_{\lambda(t)}(\dot{E}_i, E_j).
\end{align}
In particular, choosing a local coordinate system $(x^i,p^i), i=1,\dots,n$, on $T^*M$, we have $B_{ij}= \frac{\partial^2 H}{\partial p_i \partial p_j}(\lambda).$
\begin{remark}
    In fact, the derivative $\dot{\mathbf{J}}(0)$ corresponds to the Hessian $d^2_\lambda H$ evaluated at the initial point (see the proof in \cite[Proposition 15.2]{Agr}). When the setting reduces to the classical Riemannian case, we can find a canonical moving frame such that $A$ in the Jacobi equation \eqref{jeq} vanishes, $B$ equals the identity matrix representing the metric, and $R$ is exactly the sectional curvature matrix. This suggests a new way to establish curvature theory based entirely on the Hamiltonian. We refer to some interesting work in this direction. The first one is \cite{Ohta14}, where the author studies the Hamiltonian system introduced by a Lagrangian defined on the whole tangent bundle and generalizes the Bochner--Weitzenb\"ock formula, Laplacian comparison theorem, and heat flow, which cover the results in both Riemannian and Finslerian settings.
    We also refer to \cite{BR2}, where the authors extend comparison theorems under lower curvature bounds to sub-Riemannian geometry. As pointed out in their work, all these constructions depend purely on the sub-Riemannian Hamiltonian flow; thus, their arguments remain applicable in sub-Finslerian settings, although this extension was not explicitly stated in their paper.
\end{remark}

Setting $S(t) = M(t)N(t)^{-1}$ (whenever $N(t)$ is invertible), we derive the matrix Riccati equation:
\begin{equation}\label{riccati_eq}
    \dot{S} + AS + SA^T + SBS + R = 0.
\end{equation}

We now state the matrix Riccati comparison theorem.

\begin{theorem}[see \cite{BR2}, Theorem B.1]\label{RcComthm}
    Let $S_1(t)$ and $S_2(t)$ be solutions to the Riccati equations
    \begin{equation}
        \dot{S}_i + A_i S_i + S_i A_i^T + S_i B_i S_i + R_i = 0, \quad i=1,2,
    \end{equation}
    defined on a common interval $I \subseteq \mathbb{R}$, where $B_i,R_i$ are symmetric matrices, and let
    \[M_i=\begin{pmatrix}
        R_i & A_i^T\\ A_i & B_i
    \end{pmatrix},\quad i=1,2.\]
    Suppose $S_1(t_0) \ge S_2(t_0)$ at some $t_0 \in I$, and $M_1(t)\le M_2(t)$ for all $t \in I\cap[t_0,+\infty)$. Then $S_1(t) \ge S_2(t)$ for all $t \in I\cap[t_0,+\infty)$.
\end{theorem}

\subsection{Jacobian estimate}
Let $\gamma:[0,1]\to M$ be a normal geodesic from $x$ to $y$. Fix a Darboux moving frame $E_i(t), F_i(t)$ ($i=1,\dots,n$) along the corresponding extremal $\lambda(t)=e^{t\vec{H}}(\lambda_0)$. We introduce two special Jacobi matrices $\mathbf{J}^{\mathrm{V}}_s(t)$ and $\mathbf{J}^{\mathrm{H}}_s(t)$ defined by their initial conditions at a given time $s \in [0,1]$. 
Let $\mathbf{J}^{\mathrm{V}}_s(t) = \begin{pmatrix} M^{\mathrm{V}}_s(t)\\ N^{\mathrm{V}}_s(t) \end{pmatrix}$ be the Jacobi matrix satisfying the initial condition $\mathbf{J}^{\mathrm{V}}_s(s) = \begin{pmatrix} I_n\\ 0 \end{pmatrix}$. It geometrically represents the evolution of the family of Lagrangian vertical subspaces along the flow, i.e., $e^{(t-s)\vec{H}}_*\mathcal{V}_{\lambda(s)}$, where $\mathcal{V}_{\lambda(s)}=\operatorname{span}\{E_1(s),\dots,E_n(s)\}$.
Similarly, let $\mathbf{J}^{\mathrm{H}}_s(t) = \begin{pmatrix} M^{\mathrm{H}}_s(t)\\ N^{\mathrm{H}}_s(t) \end{pmatrix}$ be the Jacobi matrix satisfying the initial condition $\mathbf{J}^{\mathrm{H}}_s(s) = \begin{pmatrix} 0\\ I_n \end{pmatrix}$.

\begin{remark}\label{rm25}
    For $s_2,s_1\in [0,1]$ with $s_1<s_2$, if $\gamma(s_2)$ is not conjugate to $\gamma(s_1)$, by Lemma~\ref{cjlm}, 
    we have $\mathcal{V}_{\lambda_{s_2}}\cap e^{(s_2-s_1)\vec{H}}_*\mathcal{V}_{\lambda_{s_1}} = \{0\}$. 
    Consequently, the projection of $e^{(s_2-s_1)\vec{H}}_*\mathcal{V}_{\lambda_{s_1}}$ onto the horizontal 
    subspace $\mathcal{H}_{\lambda_{s_2}} := \operatorname{span}\{F_1(s_2), \dots, F_n(s_2)\}$ along $\mathcal{V}_{\lambda_{s_2}}$ is a linear isomorphism. Hence, $N^{\mathrm{V}}_{s_1}(s_2)$ is invertible. When it comes to the reverse direction, since the sub-Finslerian structure is irreversible in general, the curve $\exp_{\gamma(s_2)}(-t):=\pi \circ e^{-t\vec{H}}(\lambda_{s_2})$ for $t\in [0,s_2-s_1]$ is not necessarily a geodesic from $\gamma(s_2)$ to $\gamma(s_1)$. Thus, it is not trivial to determine whether $\gamma(s_1)$ is conjugate to $\gamma(s_2)$. However, thanks to the symplectomorphism property of the Hamiltonian flow, we still have $\mathcal{V}_{\lambda_{s_1}}\cap e^{(s_1-s_2)\vec{H}}_*\mathcal{V}_{\lambda_{s_2}} = \{0\}$, which ensures that the corresponding projection is still a linear isomorphism. The same argument implies that $N^{\mathrm{V}}_{s_2}(s_1)$ is also invertible.
\end{remark}

Using the notations introduced above, we obtain the following explicit representations for the Jacobi matrix $\mathbf{J}(t)$.

\begin{lemma}\label{jcblemma}
    Assume that for an intermediate time $s\in(0,1]$, $\gamma(s)$ is not conjugate to $\gamma(0)$. Then for all $t\in[0,1]$, any Jacobi matrix $\mathbf{J}(t)=\begin{pmatrix} M(t)\\ N(t) \end{pmatrix}$ satisfies the following equations:
    \begin{align}
        \mathbf{J}(t) &= \mathbf{J}^{\mathrm{V}}_s(t)N^{\mathrm{V}}_s(0)^{-1}N(0) + \mathbf{J}^{\mathrm{V}}_0(t)N^{\mathrm{V}}_0(s)^{-1}N(s),\\
        \mathbf{J}(t) &= \mathbf{J}^{\mathrm{V}}_0(t)M(0) + \mathbf{J}^{\mathrm{H}}_0(t)N(0),\\
        \mathbf{J}^{\mathrm{V}}_s(t) &= -\mathbf{J}^{\mathrm{V}}_0(t)N^{\mathrm{V}}_0(s)^{-1}N^{\mathrm{H}}_0(s)N^{\mathrm{V}}_s(0) + \mathbf{J}^{\mathrm{H}}_0(t)N^{\mathrm{V}}_s(0).
    \end{align}
\end{lemma}
\begin{proof}
    By Lemma~\ref{cjlm}, since $\gamma(s)$ is not conjugate to $\gamma(0)$, any Jacobi matrix is uniquely determined by its horizontal components $N(0)$ and $N(s)$ at two distinct times $t=0$ and $t=s$. This allows us to express $\mathbf{J}(t)$ as a linear combination of $\mathbf{J}^{\mathrm{V}}_s(t)$ and $\mathbf{J}^{\mathrm{V}}_0(t)$. We can therefore assume that $\mathbf{J}(t) = \mathbf{J}^{\mathrm{V}}_s(t)A + \mathbf{J}^{\mathrm{V}}_0(t)B$ for some constant matrices $A$ and $B$. 
    
    Evaluating the horizontal component $N(t)$ at $t=0$ and $t=s$, and noting that $N^{\mathrm{V}}_0(0)=0$ and $N^{\mathrm{V}}_s(s)=0$, we solve for $A$ and $B$:
    \begin{align}
        A &= N^{\mathrm{V}}_s(0)^{-1}N(0),\\
        B &= N^{\mathrm{V}}_0(s)^{-1}N(s),
    \end{align}
    where the invertibility of the matrices $N^{\mathrm{V}}_s(0)$ and $N^{\mathrm{V}}_0(s)$ is guaranteed by Remark~\ref{rm25}. This yields the first equation.

    Furthermore, any Jacobi matrix $\mathbf{J}(t)$ can also be uniquely expanded using the basis matrices initialized at $t=0$. By simply evaluating the initial conditions of $\mathbf{J}(t)$ at $t=0$, we straightforwardly obtain the second equation:
    \begin{equation*}
        \mathbf{J}(t) = \mathbf{J}^{\mathrm{V}}_0(t)M(0) + \mathbf{J}^{\mathrm{H}}_0(t)N(0).
    \end{equation*}

    The third equation can be deduced in the same way by applying the second equation to expand $\mathbf{J}^{\mathrm{V}}_s(t)$. 
    That is, $\mathbf{J}^\mathrm{V}_s(t)=\mathbf{J}^{\mathrm{V}}_0(t)M^\mathrm{V}_s(0)+\mathbf{J}^\mathrm{H}_0(t)N^\mathrm{V}_s(0).$
    Evaluating its horizontal component at $t=s$ and using $N^{\mathrm{V}}_s(s)=0$, we derive that $M^\mathrm{V}_s(0)=-N^\mathrm{V}_0(s)^{-1}N^\mathrm{H}_0(s)N^\mathrm{V}_s(0)$ and then complete the proof.
\end{proof}

The main result of this section is the following Jacobian estimate.

\begin{theorem}[Jacobian estimate]\label{jest}   
    Suppose that a normal geodesic $\gamma(t)=\exp_x(t d_x\phi)$ for $t\in[0,1]$ is a unique minimizing curve connecting $x$ to $y$ and contains no abnormal segments, where $\phi$ is a twice differentiable function at $x\in M$ satisfying 
    \begin{equation}
        \frac{1}{2}d^2_{SF}(x,y)=-\phi(x),\quad \text{and}\quad \frac{1}{2}d_{SF}^2(z,y)\ge-\phi(z), \quad \forall z\in M.
    \end{equation}
    Then, $y\notin \mathrm{Cut}(x)$, and the tangent map $d_xT_t:T_xM\to T_{\gamma(t)}M$ defined by $d_xT_t:=\pi_*\circ e^{t\vec{H}}_*\circ d^2_x\phi$ satisfies the following estimate for every fixed $s\in(0,1]$:
    \begin{equation}\label{eq:jest}
        \det(d_xT_t)^{1/n} \ge \left(\frac{\det N^{\mathrm{V}}_s(t)}{\det N^{\mathrm{V}}_s(0)}\right)^{1/n} + \left(\frac{\det N^{\mathrm{V}}_0(t)}{\det N^{\mathrm{V}}_0(s)}\right)^{1/n}\det(d_xT_s)^{1/n},\quad \forall t\in[0,s],
    \end{equation}
    where the determinants are computed with respect to a smooth moving frame along $\gamma$. Both terms on the right-hand side are nonnegative for $t\in[0,s]$, and the first term is strictly positive for $t\in[0,s)$. In particular, $\det(d_xT_t)>0$ for all $t\in[0,1)$.
\end{theorem}

\begin{proof}
   Let $\lambda(t)=e^{t\vec{H}}(\lambda_0)$ be the normal extremal associated with $\gamma(t)$.
    Let $E_i(t), F_i(t)$, $i=1,\dots,n$, be a Darboux moving frame along $\lambda(t)$, and denote by $X_i(t)$ the projected moving frame $\pi_*(F_i(t))$ along $\gamma$. For each $i=1,\dots,n$, let $\mathcal{J}_i(t)=e^{t\vec{H}}_*\mathcal{J}_i(0)$ be the Jacobi field initialized by $\mathcal{J}_i(0)=d^2_x\phi(X_i(0))$. In this frame, the horizontal matrix $N(t)$ represents the tangent map $d_xT_t$, namely,
    \begin{equation}
        d_xT_t(X_i(0)) = \sum_{j=1}^n N_{ji}(t)X_j(t),
    \end{equation}
    and $N(0)$ is the identity matrix $I_n$.

    For $s\in(0,1)$, Theorem~\ref{conjugatepoint} implies that $\gamma(s)$ is not conjugate to $\gamma(0)$. 
    By Remark~\ref{rm25} and Lemma~\ref{jcblemma}, we have the horizontal matrix identity:
    \begin{equation}\label{eqn}
        N(t) = N^{\mathrm{V}}_s(t)N^{\mathrm{V}}_s(0)^{-1} + N^{\mathrm{V}}_0(t)N^{\mathrm{V}}_0(s)^{-1}N(s), \quad \forall t\in[0,1].
    \end{equation}

    To proceed, we observe that for $s\in(0,1)$, $N^{\mathrm{V}}_0(s)^{-1}N(s)$ is symmetric and nonnegative definite;
    in addition, the matrix
    $N^{\mathrm{V}}_0(t)^{-1}N^{\mathrm{V}}_s(t)N^{\mathrm{V}}_s(0)^{-1}$ is symmetric and nonnegative definite for 
    $t\in(0,s]$, and $\det N^{\mathrm{V}}_0(t)>0$ 
    for $t\in(0,1)$. If $\gamma(1)$ is not conjugate to $\gamma(0)$, then the conclusion can be strengthened to 
    hold for all $s\in(0,1]$ and $t\in(0,1]$. The proof is identical to the sub-Riemannian setting; 
    see \cite[Lemma 29]{BR}, and a sketch is provided in Appendix B.

    This positivity allows us to apply the Minkowski determinant theorem to $N^{\mathrm{V}}_0(t)^{-1}N(t)\mapsto \det\left(N^{\mathrm{V}}_0(t)^{-1}N(t)\right)$, using \eqref{eqn}, yielding:
    \begin{equation}\label{je}
        \det(d_xT_t)^{1/n} \ge \left(\frac{\det N^{\mathrm{V}}_s(t)}{\det N^{\mathrm{V}}_s(0)}\right)^{1/n} + \left(\frac{\det N^{\mathrm{V}}_0(t)}{\det N^{\mathrm{V}}_0(s)}\right)^{1/n}\det(d_xT_s)^{1/n}, \quad \forall t\in(0,s].
    \end{equation}
    Both terms on the right-hand side are nonnegative. For $t=0$, since $d_xT_0=I_n$ and $N^{\mathrm{V}}_0(0)=0$, \eqref{je} also holds trivially.

    We now show that \eqref{je} is also valid for $s=1$. It suffices to prove that $\gamma(1)$ is not conjugate to $\gamma(0)$. Suppose, by contradiction, that $\gamma(1)$ is conjugate to $\gamma(0)$. Therefore applying Lemma~\ref{cjlm} and Remark~\ref{rm25}, we have $\det N^{\mathrm{V}}_0(1)=\det N^{\mathrm{V}}_1(0)=0$, while $\det N^{\mathrm{V}}_s(0)\neq 0$ for $s\in(0,1)$. Since the terms on the right-hand side of \eqref{je} are nonnegative, for fixed $t\in(0,1)$ we may drop the second term and obtain
    \begin{equation}\label{je2}
        \det(d_xT_t)^{1/n} \ge \left(\frac{\det N^{\mathrm{V}}_s(t)}{\det N^{\mathrm{V}}_s(0)}\right)^{1/n}.
    \end{equation}
Letting $s\uparrow 1$, the denominator $\det N^{\mathrm{V}}_s(0)$ tends to $\det N^{\mathrm{V}}_1(0)=0$, while $\det N^{\mathrm{V}}_s(t)$ remains nonzero for fixed $t<s$. Hence the right-hand side diverges to $+\infty$, whereas the left-hand side remains finite and independent of $s$. This contradiction shows that $\gamma(1)$ is not conjugate to $\gamma(0)$.
Therefore, $y\notin \mathrm{Cut}(x)$, and \eqref{je} holds for all $s\in(0,1]$ and $t\in[0,s]$. Finally, since both terms on the right-hand side of \eqref{je} are nonnegative for $t\in[0,s]$, and the first term is strictly positive whenever $0\le t<s\le 1$, it follows that $\det(d_xT_t)>0$ for all $t\in[0,1)$.
\end{proof}
\subsection{Failure of semiconvexity at the cut locus}
We say a function $f:M\to \mathbb{R}$ is \emph{semiconvex} at $x$ if in any set of local coordinates around $x$, the infimum 
\begin{equation}\label{radio1}
    \inf_{0<|v|<1}\frac{f(x+v)+f(x-v)-2f(x)}{|v|^2}
\end{equation}
is finite, and $f$ fails to be semiconvex if the infimum \eqref{radio1} is $-\infty$. Similarly, we say $f$ is \emph{semiconcave} if the supremum 
\begin{equation}\label{radio2}
    \sup_{0<|v|<1}\frac{f(x+v)+f(x-v)-2f(x)}{|v|^2}
\end{equation}
is finite, and $f$ fails to be semiconcave if the supremum \eqref{radio2} diverges to $+\infty$.
As a byproduct of the Jacobian estimate, we find another way to characterize the cut locus, which is closely related to the semiconvexity of the distance function, as stated below.

\begin{theorem}\label{thm_conv}
    Let $(M,\mathcal{D},F)$ be a forward ideal sub-Finslerian manifold. Then $y\in \mathrm{Cut}(x)$ if and only if the squared sub-Finslerian distance to $y$ fails to be semiconvex at $x$.
\end{theorem}

\begin{proof}
    Denote $f(z):=d^2_{SF}(z,y)$ for $z \in M$.
    First, note that for any $y\notin \mathrm{Cut}(x),$ $f$ is smooth in a neighborhood of $x$. Thus, the infimum in \eqref{radio1} is finite, which implies the semiconvexity of $f$ at $x$. 

    To prove the converse, we fix an auxiliary sub-Riemannian metric $g$ on the distribution $\mathcal{D}$ such that $(M, \mathcal{D}, g)$ 
    forms an  ideal sub-Riemannian manifold. The existence of such a metric is guaranteed by \cite[Theorem 2.8]{CJT} (see also 
    \cite[Proposition 14]{BR}). Denote $d_{SR}$ as the corresponding sub-Riemannian distance and $\tilde{f}(z):=d^2_{SR}(z,y)$. 
    Following the strategy in \cite[Section 4.1]{BR}, by \cite[Theorems 1 and 5]{CR}, $\tilde{f}$ is locally semiconcave near $x$. Given that $d_{SF}$ is locally equivalent to $d_{SR}$ (as implied by \eqref{sfsr}) and utilizing \cite[Proposition 3.3.1]{CR}, there exist local coordinates around $x$ such that 
    \begin{equation}\label{eq1}
        f(x+v)-f(x)\le p\cdot v +C|v|^2,\quad \forall |v|<1,
    \end{equation}
    for some $p\in \mathbb{R}^n, C\in \mathbb{R}.$ Suppose $f$ is semiconvex at $x$; then for some $K\in \mathbb{R}$, 
    \begin{equation}\label{eq2}
        f(x+v)+f(x-v)-2f(x)\ge K|v|^2,\quad \forall |v|<1.
    \end{equation}
    The inequalities \eqref{eq1} and \eqref{eq2} allow us to construct a function $\phi:M\to\mathbb{R}$ that is twice differentiable at $x$, satisfying $\frac{1}{2}f(z)\ge \phi(z)$ for $z\in M$ and $\frac{1}{2}f(x)=\phi(x)$. Finally, via Theorem~\ref{jest}, we conclude $y\notin \mathrm{Cut}(x)$, which completes the proof.
\end{proof}

\section{Optimal transport theory}
Let us fix a smooth positive measure $\mathfrak{m}$ on $M$. In local
coordinates $(x^1,\dots,x^n)$, it can be written as
\[
d\mathfrak{m}=f(x)\,dx^1\dots dx^n,
\]
where $f$ is a smooth positive function. For linearly
independent vector fields $X_1,\dots,X_n$, written locally as
\[
X_i=\sum_{j=1}^n X_i^j\frac{\partial}{\partial x^j},
\]
we define
\[
\mathfrak{m}(X_1,\dots,X_n)(x)
:=
f(x)\left|\det\bigl(X_i^j(x)\bigr)\right|.
\]
This quantity is independent of the choice of local coordinates and
represents the infinitesimal $\mathfrak{m}$-volume of the parallelepiped
spanned by $X_1(x),\dots,X_n(x)$.

Let $\mathcal{P}_c(M)$ represent the space of probability 
measures on $M$ with compact supports, and let $\mathcal{P}^{ac}_c(M)$ be the subspace of those 
absolutely continuous ones w.r.t. $\m.$ Let $\mu_0,\mu_1$ be two probability measures on $M$. 
In this and following sections, we denote the forward and backward quadratic cost functions as 
$c(x,y)=d^2_{SF}(x,y)$ and $\bar{c}(x,y)=d^2_{SF}(y,x)$, respectively. 

The classical Monge problem asks us to find a transport map $T:M\to M$ pushing $\mu_0$ forward to 
$\mu_1$ (i.e., $T_{\sharp}\mu_0=\mu_1$) that minimizes the transport cost:
\begin{equation}
    \int_M c(x,T(x))d\mu_0(x).
\end{equation}
However, as we know, the Monge problem is not always well-posed. A relaxed version proposed by 
Kantorovich suggests searching for optimal transport plans instead of optimal transport maps. 
A plan is a probability measure $\pi$ on $M\times M,$ whose marginals are $\mu_0$ and $\mu_1,$ 
and it seeks to attain the minimal transport cost:
\begin{equation}\label{ckp}
    C(\mu_0,\mu_1):=\min\limits_{\pi\in \Pi(\mu_0,\mu_1)}\left\{\int_{M\times M}c(x,y)d\pi(x,y)\right\},
\end{equation}
where $\Pi(\mu_0,\mu_1)$ denotes the set of all transport plans with marginals $\mu_0$ and $\mu_1$.
Unlike the Monge problem, the Kantorovich problem always admits a solution provided that 
$C(\mu_0,\mu_1)$ is finite (see \cite[Theorem 5.10]{Vi}). Hence, the next goal is to explore when, on
a sub-Finslerian manifold, the optimal transport plan can be expressed by a transport map.
For this target, we introduce the following notions.

\begin{definition}[$c$-transform]
    Let $X,Y\subset M$ be two compact sets. For any $\varphi:X\to \mathbb{R}\cup\{-\infty\}$, we define the \textit{$c$-transform} $\varphi^c:Y\to \mathbb{R}\cup \{-\infty\}$ relative to $(X,Y)$ by 
    \begin{equation}
            \varphi^c(y):=\inf\limits_{x\in X}\{c(x,y)-\varphi(x)\}.
    \end{equation}
    Similarly, the $\bar c$-transform of $\psi:Y\to \mathbb{R}\cup\{-\infty\}$ is defined by 
    \begin{equation}
        \psi^{\bar c}(x):=\inf\limits_{y\in Y}\{\bar{c}(y,x)-\psi(y)\}=\inf\limits_{y\in Y}\{c(x,y)-\psi(y)\}.
    \end{equation}
\end{definition}

\begin{definition}[$c$-concavity]
    A function $\varphi:X\to \mathbb{R}\cup\{-\infty\}$ is said to be \textit{$c$-concave} if it is not identically $-\infty$ and if there exists a function $\psi:Y\to \mathbb{R}\cup\{-\infty\}$ such that $\psi^{\bar c}=\varphi.$ Namely,
    \begin{equation}
        \varphi(x) = \inf\limits_{y\in Y}\{c(x,y)-\psi(y)\}.
    \end{equation}
\end{definition}

\begin{definition}[$c$-superdifferential]
    Suppose $\varphi$ is $c$-concave. The \textit{$c$-superdifferential} of $\varphi$ at a point $x\in X$ is:
    \begin{equation}
        \partial^c\varphi(x):=\{y\in Y \mid \varphi(x)=c(x,y)-\varphi^c(y)\}.
    \end{equation}
    Similarly,
    \begin{equation}
        \partial^{\bar{c}}\psi(y):=\{x\in X \mid \psi(y)=c(x,y)-\psi^{\bar c}(x)\}.
    \end{equation}
\end{definition}

The famous Kantorovich duality, 
which builds a bridge between optimal transport plans and certain $c$-concave functions, 
called Kantorovich potentials, is stated below.

\begin{theorem}[see Theorem 5.10 in \cite{Vi}]\label{bm1}
    Let $(M,\mathcal{D},F)$ be a sub-Finslerian manifold, and let $\mu_0,\mu_1\in \mathcal{P}_c(M).$ Then the following duality holds:
    \begin{equation}
            \min\limits_{\pi\in \Pi(\mu_0,\mu_1)}\int_{M\times M}c(x,y)d\pi(x,y)=\sup\limits_{\varphi}\left\{\int_M \varphi(x)d\mu_0(x) + \int_M \varphi^c(y)d\mu_1(y)\right\}.
    \end{equation} 
    Moreover, there exists a continuous $c$-concave function $\varphi$ such that for any $\pi\in \Pi(\mu_0,\mu_1)$, $\pi$ is an optimal plan if and only if $\supp(\pi)\subset \{(x,y) \in M \times M \mid y \in \partial^c\varphi(x)\}.$
\end{theorem}

In Finsler geometry (and Riemannian geometry as a special example), the $c$-superdifferential of Kantorovich potential $\varphi$ can be represented via the gradient and exponential map, i.e., 
$\partial^c\varphi(x)=\{\exp_x(-\frac{1}{2}d_x\varphi)\}$ provided that $\mu_0\in\mathcal{P}^{ac}_c(M)$. However, this good property fails to hold in general 
sub-Finslerian setting since the regularity of the distance function along abnormal geodesics is delicate to study. 
Hence, here we only study optimal transport on forward ideal sub-Finslerian manifolds, since the squared distance function in this context 
is locally semi-concave outside the diagonal as we showed in Theorem~\ref{lip}. For general regularity cases, we refer to \cite{FR}. 

 Before we state sub-Finslerian Brenier-McCann theorem, let us define the moving set $\mathcal{M}^{\varphi}$ and the static set $\mathcal{S}^\varphi$ respectively as:
\begin{equation}
    \begin{aligned}
        \mathcal{M}^{\varphi}&=\{x \in M \mid x\notin \partial^c\varphi(x)\};\\
        \mathcal{S}^\varphi&=\{x \in M \mid x\in \partial^c\varphi(x)\}.
    \end{aligned}
\end{equation}

\begin{theorem}[Brenier-McCann Theorem]\label{ot}
    Let $(M,\mathcal{D},F)$ be a forward ideal sub-Finslerian manifold and let $\mu_0\in \mathcal{P}^{ac}_c(M)$ and $\mu_1\in\mathcal{P}_c(M)$ be two probability measures on $M$. Let $\varphi$ be the potential provided in Theorem \ref{bm1}. Then:
    \begin{itemize}
        \item[(1)] $\mathcal{M}^{\varphi}$ is open and $\varphi$ is locally semiconcave in a neighborhood of 
        $\mathcal{M}^{\varphi}\cap \supp(\mu_0).$ Moreover, for $\mu_0$-a.e. $x\in \mathcal{M}^{\varphi}\cap \supp(\mu_0)$, $\partial^c\varphi(x)$ is a singleton.
        \item[(2)] For $\mu_0$-a.e. $x\in \mathcal{S}^{\varphi}\cap \supp(\mu_0)$, $\partial^c\varphi(x)=\{x\}.$
    \end{itemize}
    In particular, there exists a unique optimal transport map $T$ defined $\mu_0$-a.e. by 
    \begin{equation}\label{opte}
        T(x):=\begin{cases}
            \exp_x(-\frac{1}{2}d_x\varphi), &x\in \mathcal{M}^{\varphi}\cap \supp(\mu_0),\\
            x,& x\in \mathcal{S}^{\varphi}\cap \supp(\mu_0).
        \end{cases}
    \end{equation}
\end{theorem}
\begin{proof}
    Let us prove (1) first. Since $\mathcal{M}^{\varphi}=\{x \mid x\notin \partial^c\varphi(x)\}=\{x \mid \varphi(x)+\varphi^c(x)<0\}$, by the continuity of $\varphi$ and $\varphi^c$, $\mathcal{M}^{\varphi}$ is open. Now we need to show $\varphi$ is locally semiconcave. Let $x\in \mathcal{M}^{\varphi}\cap \supp(\mu_0)$; then for any $y\in \partial^c\varphi(x)$, there exists $r>0$ such that $d_{SF}(x,y)>r.$
    Furthermore, since $\partial^c\varphi$ is closed in $M\times M,$ there exists a neighborhood $\mathcal{V}_x\subset \mathcal{M}^{\varphi}\cap \supp(\mu_0)$ of $x$ and $r>0$ such that  
    \begin{equation}
        d_{SF}(z,w)>r,\quad \forall z\in \mathcal{V}_x,\ \forall w\in\partial^c\varphi(z).
    \end{equation}
    Define
    \begin{equation}
        \varphi_{x,r}(z):=\inf\{c(z,y)-\varphi^c(y) \mid y\in\supp(\mu_1), d_{SF}(z,y)>r\}.
    \end{equation}
    Via Theorem \ref{lip}, $c(z,y) = d^2_{SF}(z,y)$ is locally semiconcave outside the 
    diagonal; hence $\varphi_{x,r}(z)$ is locally semiconcave and a fortiori locally Lipschitz. 
    Since $\varphi=\varphi_{x,r}$ in $\mathcal{V}_x$, the same is true for $\varphi$.
    Moreover, $\partial^c\varphi(x)$ is a singleton $\mu_0$-a.e. In fact, for any $y\in\partial^c\varphi(x)$ and $z\in M$, 
    \begin{equation}
        \varphi^c(y)=c(x,y)-\varphi(x)\le c(z,y)-\varphi(z),
    \end{equation}
    which implies 
    \begin{equation}
        d^2_{SF}(z,y)\ge \varphi(z)-\varphi(x)+d^2_{SF}(x,y).
    \end{equation}
    Hence, by Proposition \ref{prop1} and Rademacher's theorem, we obtain $y=\exp_{x}(-\frac{1}{2}d_x\varphi)$ for $\mu_0$-a.e. $x\in \mathcal{M}^{\varphi}\cap \supp(\mu_0)$, which means $y$ is uniquely specified.

    Now let us consider the static set. It is sufficient to show the result for points lying in an open set $\mathcal{V}\subset M$.
    Choose local coordinates on $\mathcal{V}$ such that the generating frame $X_1,\dots, X_n$ on $\mathcal{V}$ can be expressed by
    \begin{equation}
        X_i=\frac{\partial}{\partial x^i}+\sum_{j} a_{ij}\frac{\partial}{\partial x^j}, \quad \text{where } a_{ij}\in C^{\infty}.
    \end{equation}
    Using \cite[Theorem 3.2]{MC} and the equivalence \eqref{sfsr}, we have the local expansion:
    \begin{equation}\label{est}
        \varphi(z)-\varphi(x)-\sum_{i=1}^k (z_i-x_i) X_i\varphi(x)=o(d_{SF}(x,z)),\quad z\in\mathcal{V}.
    \end{equation}
    Fix $x\in \mathcal{S}^\varphi\cap \mathcal{V}$ at where $\varphi$ is differentiable, then $X_i\varphi(x)=0$. As a matter of fact, denote the integral curves of the vector fields $X_i$ by $\gamma^x_i:(-\epsilon,\epsilon)\to M$ such that 
    \begin{equation}
        \begin{cases}
            \dot{\gamma}^x_i(t)=X_i(\gamma^x_i(t)) \quad \forall t\in(-\epsilon,\epsilon),\\
            \gamma^x_i(0)=x.
        \end{cases}
    \end{equation}
    Since $x\in\partial^c\varphi(x),$ we have $\varphi(\gamma^x_i(t))\le \varphi(x)+c(\gamma^x_i(t),x)\le \varphi(x)+C t^2$, which implies the horizontal derivative $X_i\varphi(x)=0.$

    Suppose there is a distinct point $y\in \partial^c\varphi(x)\setminus\{x\}.$ Then the function $z\mapsto \varphi(z)-d^2_{SF}(z,y)$ attains its maximum at $x$.
    Set $\gamma_{x,y}$ as a minimizing normal geodesic connecting $x$ to $y$. Then 
    \begin{equation}
        \varphi(\gamma_{x,y}(t))-d^2_{SF}(\gamma_{x,y}(t),y)\le \varphi(x)-d^2_{SF}(x,y).
    \end{equation}
    Then by \eqref{est} and $d_{SF}(\gamma_{x,y}(t),y) = (1-t)d_{SF}(x,y)$, we deduce that 
    \begin{equation}
        \begin{aligned}
            o(t d_{SF}(x,y)) &= o(d_{SF}(x,\gamma_{x,y}(t))) = \varphi(\gamma_{x,y}(t))-\varphi(x)\\
            &\le d^2_{SF}(\gamma_{x,y}(t),y)-d^2_{SF}(x,y)\\
            &= (1-t)^2 d^2_{SF}(x,y) - d^2_{SF}(x,y)\\
            &= -2t d^2_{SF}(x,y)+t^2 d^2_{SF}(x,y).
        \end{aligned}
    \end{equation}
    Letting $t \to 0$, we get $0 \le -2 d^2_{SF}(x,y)$, which implies $x=y$, leading a contradiction.

    By Theorem \ref{bm1}, for any optimal plan $\pi\in\Pi(\mu_0,\mu_1),$ its support is contained in $\partial^c\varphi$, and we have shown that for $\mu_0$-a.e. $x$, $\partial^c\varphi(x)$ is a singleton. The existence and uniqueness of the optimal transport map can be immediately obtained. Moreover, for $\mu_0$-a.e. $x\in \mathcal{M}^{\varphi}\cap \supp(\mu_0),$ $T(x)=\exp_x(-\frac{1}{2}d_x\varphi)$ and for $\mu_0$-a.e. $x\in \mathcal{S}^\varphi\cap \supp(\mu_0)$, $T(x)=x.$ The uniqueness follows from the construction.
\end{proof}

The sub-Finslerian Brenier--McCann theorem provides a way to construct a Wasserstein geodesic connecting two smooth measures. 
For $t\in[0,1]$, let $T_t(x)$ be defined by
\begin{equation}\label{opte2}
        T_t(x):=\begin{cases}
            \exp_x\left(-\frac{t}{2}d_x\varphi\right), &x\in \mathcal{M}^{\varphi}\cap \supp(\mu_0),\\
            x,& x\in \mathcal{S}^{\varphi}\cap \supp(\mu_0),
        \end{cases}
    \end{equation}
so that $t\mapsto T_t(x)$ traces out the minimizing geodesic from $x$ to $T_1(x)=T(x)$, where $T$ is the map constructed in \eqref{opte}. We now show that the intermediate measures along this transport are absolutely continuous.
\begin{theorem}\label{ot2}
    Under the assumptions of Theorem \ref{ot}, there exists a unique Wasserstein geodesic joining $\mu_0$ to $\mu_1$, given by $\mu_t=(T_t)_{\sharp}\mu_0$ for $t\in[0,1],$ where $T_t$ is the optimal transport given by \eqref{opte2}. Moreover, $\mu_t\in \mathcal{P}^{ac}_c(M)$ for all $t\in[0,1).$ 
\end{theorem}

\begin{proof}
    We adapt the standard argument (see e.g., \cite[Section 6.3]{FR}). By the basic representation theorem (see \cite[Corollary 7.22]{Vi}) and Theorem \ref{ot}, the Wasserstein geodesic must take the form $\mu_t=(T_t)_\sharp\mu_0,$ and then uniqueness follows.

    Next we deal with the absolute continuity of $\mu_t$. Since the sub-Finslerian distance may fail to be locally semiconcave on the diagonal, we split $\mathcal{M}^\varphi\cap \supp(\mu_0)$ into subsets bounded away from the diagonal:
    \begin{equation}
        A_k:=\left\{x\in \mathcal{M}^{\varphi}\cap \supp(\mu_0) \mid d_{SF}(x,T(x))>\frac{1}{k}\right\}.
    \end{equation}
    It is sufficient to see the absolute continuity of $\mu^k_t:=\mu_t|_{T_t(A_k)},$ since the situation on the static set is trivial.

    Define
    \begin{equation}
        \varphi_{k,t}(z):=\inf\left\{\frac{d^2_{SF}(x,z)}{t}-\varphi(x) \mid x\in \supp(\mu_0), d_{SF}(x,z)>\frac{t}{k}\right\},
    \end{equation}
    and then local semiconcavity of $\varphi_{k,t}(z)$ is an immediate consequence of the local semiconcavity of $d^2_{SF}(x,z)$ outside the diagonal.
    In addition, by definition,
    \begin{equation}\label{ineq:0}
        \varphi_{k,t}(z)\le\frac{d^2_{SF}(x,z)}{t}-\varphi(x)\quad \text{on } \left\{z \mid d_{SF}(x,z)>\frac{t}{k}\right\},  
    \end{equation}
    where we claim the equality holds at $z=T_t(x)$ for $x\in A_k.$ 
    
    If $\varphi_{k,t}(T_t(x))<\frac{d^2_{SF}(x,T_t(x))}{t}-\varphi(x),$ this strict inequality means there exists an $x'$ such that 
    \begin{equation}\label{ineq:1}
        \frac{d^2_{SF}(x',T_t(x))}{t}-\varphi(x')<\frac{d^2_{SF}(x,T_t(x))}{t}-\varphi(x).
    \end{equation}
    On the other hand, squaring the triangle inequality $d_{SF}(x',T(x)) \le d_{SF}(x',T_t(x)) + d_{SF}(T_t(x),T(x))$ and using the fundamental inequality, we derive that 
    \begin{equation}\label{ineq:2}
        d^2_{SF}(x',T(x))\le \frac{1}{t}d^2_{SF}(x',T_t(x))+\frac{1}{1-t}d^2_{SF}(T_t(x),T(x)),
    \end{equation}
    with exact equality holding when $x'$ is replaced by $x$. Subtracting this exact equality for $x$ from \eqref{ineq:2}, we have
    \begin{equation}
        d^2_{SF}(x',T(x))-d_{SF}^2(x,T(x))\le \frac{d^2_{SF}(x',T_t(x))}{t}-\frac{d^2_{SF}(x,T_t(x))}{t} < \varphi(x')-\varphi(x),
    \end{equation}
    which implies $c(x',T(x))-\varphi(x') < c(x,T(x))-\varphi(x) = \varphi^c(T(x))$, contradicting the definition of $\varphi^c(T(x))$. 

    To apply Proposition \ref{prop1}, we need to prove $\varphi_{k,t}(z)$ is differentiable at $z=T_t(x).$ For this purpose, we construct another auxiliary function 
    \begin{equation}
        \tilde{\varphi}_{k,t}(z):=\inf\left\{\frac{d^2_{SF}(z,y)}{1-t}-\varphi^c(y) \mid y\in\supp(\mu_1), d_{SF}(z,y)>\frac{1-t}{k}\right\}.
    \end{equation} 
    Using \eqref{ineq:2} again, one can check that $\tilde{\varphi}_{k,t}(z)$ is also a locally semiconcave 
    function, and moreover $\varphi_{k,t}(z)+\tilde{\varphi}_{k,t}(z)\ge 0$ for any $z\in M$, with exact 
    equality holding $\mu_t$-a.e. on $T_t(A_k).$ Applying \cite[Theorem A.19]{FF}, the differential 
    $d_z\varphi_{k,t}$ exists for $\mu_t$-a.e. $z \in T_t(A_k)$, and moreover the map 
    $z\mapsto d_z\varphi_{k,t}$ is locally Lipschitz on $T_t(A_k).$ Hence applying Proposition \ref{prop1} 
    to the reverse sub-Finslerian distance $\bar{d}_{SF}(x,y):=d_{SF}(y,x)$ and using \eqref{ineq:0}, we have 
    \begin{equation}
        x=\overline{\exp}_{T_t(x)}\left(-\frac{t}{2}d_{T_t(x)}\varphi_{k,t}\right)
    \end{equation}
    for $\mu_0$-a.e. $x\in A_k,$ where $\overline{\exp}$ denotes the exponential map of the reverse sub-Finslerian metric $\bar{F}(v):=F(-v).$ Set $\Psi_{t,k}(z)=\overline{\exp}_z(-\frac{t}{2}d_z\varphi_{k,t})$, which is locally Lipschitz on $\supp(\mu_t)\cap T_t(A_k).$ We then deduce that $(\Psi_{t,k})_\sharp(\mu^k_t)=\mu_0|_{A_k}.$ Due to the absolute continuity of $\mu_0$ and the local Lipschitz regularity of $\Psi_{t,k},$ the measure $\mu^k_t$ cannot contain any singular parts, completing the proof.
\end{proof}

Concerning the regularity of transport, we refer to \cite[Section 3.3]{FR} and \cite[Section 5.1]{BR} for a detailed discussion of the sub-Riemannian setting, where the conclusions can be paralleled in the sub-Finslerian setting presented here. Although we do not pursue this here, we still give some basic discussion.
For $\mu_0$-a.e. $x$ in the moving set, from the construction in Theorem~\ref{ot}, the transport is given by the $c$-superdifferential of a semiconcave function $\varphi$; therefore it is differentiable for 
$\mu_0$-a.e. $x\in \mathcal{M}^{\varphi}$, with differential given by
\begin{equation}
    dT_t(x)=d\exp\left(-\frac{t}{2}d_x\varphi\right)=\pi_*\circ e^{t\vec{H}}_*\circ \left(-\frac{1}{2}d^2_x\varphi\right).
\end{equation}
Furthermore, by Theorem~\ref{je}, for $\mu_0$-a.e. $x\in \mathcal{M}^{\varphi}$, $T(x)\notin\Cut{x}$ and $\det(d_xT_t)>0$ for all $t\in [0,1)$. For points in the static set, although the regularity of the transport is worse, $T_t$ fixes $x$ for $\mu_0$-a.e. $x\in \mathcal{S}^\varphi$, and this case is straightforward to handle.

We have the following Monge-Amp\`ere equation.
\begin{theorem}\label{maeq}
    Under the assumptions of Theorem~\ref{ot}, assume further that $\mu_1\in\mathcal{P}^{ac}_c(M).$
    Let $\mu_t$ be the Wasserstein geodesic joining $\mu_0$ and $\mu_1$, and let $T_t$ be the optimal transport map from $\mu_0$ to $\mu_t$ for $t\in [0,1]$, so that $\mu_t=(T_t)_{\sharp}\mu_0 \in \mathcal{P}^{ac}_c(M).$ Setting $\mu_t=\rho_t \m$, we have for $\mu_0$-a.e. $x\in M,$
    \begin{equation}
        \frac{\rho_0(x)}{\rho_t(T_t(x))}=\begin{cases}
            \det(d_xT_t)\frac{\m(X_1(t),\dots,X_n(t))}{\m(X_1(0),\dots,X_n(0))},\quad & x\in\mathcal{M}^\varphi\cap \supp(\mu_0)\\
            1, &x\in \mathcal{S}^\varphi\cap \supp(\mu_0),
        \end{cases}
    \end{equation}
    where $\{X_1(t),\dots,X_n(t)\}$ is a smooth moving frame along the geodesic $t\mapsto T_t(x)$, and the determinant of the linear tangent map $d_xT_t:T_xM\to T_{T_t(x)}M$ is computed with respect to this frame, i.e.,
    \begin{equation}
        d_xT_t(X_i(0))=\sum_{j=1}^n N_{ji}(t)X_j(t),\quad \det(d_xT_t):=\det N(t).
    \end{equation}
\end{theorem}
\begin{proof}

We adapt the standard argument in \cite[Section 6.4]{FR}.
As in the proof of Theorem~\ref{ot}, since $\mu_t\in\mathcal{P}^{ac}_c(M)$, there exists an optimal transport map $S_t$ minimizing the reversed cost $\int_M \bar{c}(x,S_t(x))\,d\mu_t(x)$, such that $(S_t)_\sharp \mu_t=\mu_0$. Indeed, one can check that for $\mu_t$-a.e.\ $x\in \supp(\mu_t)$,
    \begin{equation}
        S_t(x):=\begin{cases}
            \overline{\exp}_x\left(-\frac{t}{2}d_x\varphi^c\right), &x\in \mathcal{M}^{\varphi^c}\cap \supp(\mu_t),\\
            x,& x\in \mathcal{S}^{\varphi^c}\cap \supp(\mu_t),
        \end{cases}
    \end{equation}
    where $\overline{\exp}$ denotes the exponential map of the reverse sub-Finslerian metric $\bar{F}$. 
    Moreover, by the construction of the optimal transport maps $T_t$ and $S_t$, and noting that $x\in \partial^{\bar c}\varphi^c(y)$ if and only if $y\in \partial^c\varphi(x)$ (see \cite[Lemma 4.4]{Ohta}), we have $S_t\circ T_t(x)=x$ for $\mu_0$-a.e. $x\in M$.
    Therefore, $T_t$ is $\mu_0$-a.e. injective. 
    Applying the change-of-variables formula, for $\mu_0$-a.e. $x \in \mathcal{M}^\varphi$ we have 
     \begin{equation}
\begin{aligned}
\rho_0(x)\m(X_1(0),...,X_n(0))&=\rho_t(T_t(x))\m(d_xT_t(X_1(0)),...,d_xT_t(X_n(0)))\\
&=\rho_t(T_t(x))\m\left(\sum_{j=1}^n N_{j1}(t)X_j(t),...,\sum_{j=1}^n N_{jn}(t)X_j(t)\right)\\
&=\rho_t(T_t(x))\det(d_xT_t)\m(X_1(t),...,X_n(t)).
\end{aligned}
\end{equation}
    For $x \in \mathcal{S}^\varphi$, the transport map is the identity, yielding a ratio of $1$. Thus, the conclusion follows.
\end{proof}

\section{Interpolation inequality}
In this section, we first introduce the sub-Finslerian distortion coefficients, and utilize the Jacobian matrices to give a precise expression for the distortion coefficients. Then we conclude the proof of Theorem~\ref{interpolation}.
\subsection{Sub-Finslerian distortion coefficients}
We define the sub-Finslerian distortion coefficients as follows.
\begin{equation}\label{defdis}
    \begin{aligned}
        \beta^<_t(x,y)&:=\lim\limits_{r\to 0}\frac{\m(Z_t(x,B^+_r(y)))}{\m(B^+_r(y))},\\
        \beta^>_t(x,y)&:=\lim\limits_{r\to 0}\frac{\m(Z_t(B^-_r(x),y))}{\m(B^-_r(x))},
    \end{aligned}
\end{equation}
where $Z_t(x,y)=\{\gamma(t) \mid \gamma:[0,1]\to M \text{ is a minimizing geodesic connecting } x \text{ and } y \}$.

\begin{definition}[Forward/backward geodesic dimension]
    Let $(M,\mathcal{D},F)$ be a sub-Finslerian manifold equipped with a smooth measure $\m$. For any $p\in M$ and $s>0$, define
    \begin{equation}
        \begin{aligned}
            C^+_s(p)&:=\sup\left\{\limsup\limits_{t\to 0}\frac{1}{t^s}\frac{\m(Z_t(p,\Omega))}{\m(\Omega)} : \text{ $\Omega\subset M\setminus \mathrm{Cut}(p)$ is a Borel bounded set, } \m(\Omega)\in(0,\infty)\right\},\\
            C^-_s(p)&:=\sup\left\{\limsup\limits_{t\to 0}\frac{1}{t^s}\frac{\m(Z_t(\Omega,p))}{\m(\Omega)} : \text{ $\Omega\subset M\setminus \mathrm{Cut}(p)$ is a Borel bounded set, } \m(\Omega)\in(0,\infty)\right\}.
        \end{aligned}
    \end{equation}   
    We define the \textit{forward/backward geodesic dimension} of $(M,d_{SF},\m)$ as the non-negative real numbers:
    \begin{equation}
        \begin{aligned}
            \mathcal{N}^+(p)&:=\inf\{s>0 \mid C^+_s(p)=+\infty\},\\
            \mathcal{N}^-(p)&:=\inf\{s>0 \mid C^-_s(p)=+\infty\}.
        \end{aligned}
    \end{equation} 
\end{definition}

Although the definition of distortion coefficients is formally similar to that in the Riemannian and Finslerian settings, their geometric properties differ substantially. In sub-Riemannian geometry, the Ball--Box theorem \cite[Corollary 2.1]{Jea} implies that the geodesic dimension is generally strictly larger than the topological dimension. In view of the local equivalence \eqref{sfsr} between the sub-Riemannian and sub-Finslerian distances, the same conclusion remains valid in the sub-Finslerian setting.

\begin{proposition}[see {\cite[Proposition 4.20]{MR}}]\label{geodim} 
Let $(M,\mathcal{D},F)$ be an $n$-dimensional sub-Finslerian manifold. Then we have 
$\mathcal{N}^+(p)\ge n$ and $\mathcal{N}^-(p)\ge n$. 
Equality holds if and only if the distribution has full rank at $p.$ 
In particular, 
\begin{align*} 
\beta^<_t(x,x)&\le t^{k(x)},\quad \forall t\in[0,1],\\ 
\beta^>_t(x,x)&\le (1-t)^{k(x)},\quad \forall t\in[0,1], 
\end{align*} 
where $k(x)\ge n$. 
\end{proposition} 

Recalling the notation introduced at the beginning of Section 3.3, 
we now give an explicit expression for the distortion coefficients.
\begin{lemma}[Computation of the distortion coefficients]\label{lm:dc} 
Let $x,y\in M$, with $y\notin \mathrm{Cut}(x).$ Let $X_1,...,X_n$ be a smooth moving frame along the unique geodesic from $x$ to $y$. Then we have 
\begin{equation} 
\begin{aligned} 
\beta^<_t(x,y)&=\frac{\det N^{\mathrm{V}}_0(t)}{\det N^{\mathrm{V}}_0(1)}\frac{\m(X_1(t),...,X_n(t))}{\m(X_1(1),...,X_n(1))},\quad\forall t\in[0,1],\\ 
\beta^>_t(x,y)&=\frac{\det N^{\mathrm{V}}_1(t)}{\det N^{\mathrm{V}}_1(0)}\frac{\m(X_1(t),...,X_n(t))}{\m(X_1(0),...,X_n(0))},\quad\forall t\in[0,1]. 
\end{aligned} 
\end{equation} 
Moreover, $\beta^<_{t}(x,y)$ and $\beta^>_{t}(x,y)$ are both strictly positive for $t\in(0,1]$. 
\end{lemma} 
\begin{proof} 
For $t=0,$ both sides of the first equation vanish; hence, it suffices to consider $t>0.$ 
Since $y\notin \mathrm{Cut}(x),$ there exists a small open set $A_r\subset T^*_xM$ such that 
\begin{equation} 
\exp^t_x:=\exp_x(t\cdot): A_r \to Z_t(x,B^+_r(y)) 
\end{equation} 
is a diffeomorphism. In particular, we have 
\begin{equation} 
\beta^<_t(x,y)=\lim\limits_{r\to 0}\frac{\m(Z_t(x,B^+_r(y)))}{\m(B^+_r(y))}=\lim\limits_{r\to 0}\frac{\int_{A_r}(\exp^t_x)^* \m}{\int_{A_r}(\exp^1_x)^* \m}=\frac{(\exp^t_x)^* \m|_{\lambda_0}}{(\exp^1_x)^* \m|_{\lambda_0}}. 
\end{equation} 
To compute this, we lift the frame $X_1,...,X_n$ to a Darboux moving frame  $E_1,...,E_n,$  $F_1,...,F_n$.
Evaluating $(\exp^t_x)^*\m$ on $E_1,...,E_n$, 
we obtain
\begin{equation} 
\begin{aligned} 
(\exp^t_x)^*\m(E_1(0),...,E_n(0))&=\m(\pi_*\circ e^{t\vec{H}}_*E_1(0),...,\pi_*\circ e^{t\vec{H}}_*E_n(0))\\ 
&=\m\left(\sum_{j=1}^n (N^{\mathrm{V}}_0)_{j1}(t)X_j(t),...,\sum_{j=1}^n (N^{\mathrm{V}}_0)_{jn}(t)X_j(t)\right)\\ 
&=\det N^{\mathrm{V}}_0(t)\m(X_1(t),...,X_n(t)), 
\end{aligned} 
\end{equation} 
where the second equality follows from the definition of $N^{\mathrm{V}}_0(t)$. 
Hence, 
\begin{equation} 
\beta^<_t(x,y)=\frac{\det N^{\mathrm{V}}_0(t)\m(X_1(t),...,X_n(t))}{\det N^{\mathrm{V}}_0(1)\m(X_1(1),...,X_n(1))}. 
\end{equation} 
By the same argument, let $\overline{\exp}$ be the exponential map with respect to the reverse sub-Finslerian metric $\bar{F}.$ Then there exists an open set $\bar{A}_r\subset T^*M$ such that 
\begin{equation} 
\overline{\exp}^t_y: \bar{A}_r \to Z_t(B^-_r(x),y) 
\end{equation} 
is a diffeomorphism, where $\overline{\exp}^t_y=\pi\circ e^{t\overleftarrow{H}}|_y$, and $\overleftarrow{H}$ is the Hamiltonian vector field with respect to the reverse Hamiltonian $\bar{H}(\lambda):=H(-\lambda)$. 
Now, let $\bar{\gamma}(t)$ be the geodesic connecting $y$ and $x$ with respect to the reverse sub-Finslerian metric $\bar F.$ Then $\bar{\gamma}(t)=\gamma(1-t)$ and the corresponding normal lift is $\bar{\lambda}(t)=-\lambda(1-t).$ Consider the involution map $\iota :T^*M\to T^*M$ given by $\iota(\lambda)=-\lambda$, and set $\bar{X}_i(t)=-\iota_*X_i(1-t)$ and $\bar{E}_i(t):=-\iota_*E_i(1-t).$ Since the reverse Hamiltonian is $\bar{H}=H\circ \iota,$ the corresponding Hamiltonian flow satisfies: 
\begin{equation} 
e^{t\overleftarrow{H}}_*=\iota_*\circ e^{-t\vec{H}}_*\circ\iota_*. 
\end{equation} 
It follows that 
\begin{equation} 
e^{t\overleftarrow{H}}_*\bar{E}_i(0) = \iota_*\circ e^{-t\vec{H}}_*(- E_i(1)) = -\iota_*\circ e^{-t\vec{H}}_* E_i(1), \quad \text{for } i=1,...,n, 
\end{equation} 
where we used the fact that $\iota^2=\mathrm{Id}.$ 

Because $\pi \circ \iota = \pi$, applying the differential $\pi_*$ gives 
$\pi_* \circ \iota_* = \pi_*$. Thus, we conclude
\begin{equation} 
\begin{aligned} 
\beta^>_t(x,y)&=\lim\limits_{r\to 0}\frac{\int_{\bar{A}_r}(\overline{\exp}^{1-t}_y)^*\m}{\int_{\bar{A}_r}(\overline{\exp}^{1}_y)^*\m}\\ 
&=\frac{\m\Big(\pi_*\circ e_*^{(1-t)\overleftarrow{H}}\bar{E}_1(0),...,\pi_*\circ e_*^{(1-t)\overleftarrow{H}}\bar{E}_n(0)\Big)}{\m\Big(\pi_*\circ e_*^{\overleftarrow{H}}\bar{E}_1(0),...,\pi_*\circ e_*^{\overleftarrow{H}}\bar{E}_n(0)\Big)}\\ 
&=\frac{\m\Big(-\pi_*\circ e_*^{(t-1)\vec{H}}E_1(1),...,-\pi_*\circ e_*^{(t-1)\vec{H}}E_n(1)\Big)}{\m\Big(-\pi_*\circ e_*^{-\vec{H}}E_1(1),...,-\pi_*\circ e_*^{-\vec{H}}E_n(1)\Big)}\\ 
&=\frac{\det N^{\mathrm{V}}_1(t)}{\det N^{\mathrm{V}}_1(0)}\frac{\m(X_1(t),...,X_n(t))}{\m(X_1(0),...,X_n(0))},\quad \forall t\in[0,1). 
\end{aligned} 
\end{equation} 
This completes the proof. 
\end{proof}

\subsection{Proof of Theorem \ref{interpolation}}

The interpolation inequality is now ready to be proved from Theorem \ref{jest} and the computation formula for the distortion coefficients in Lemma \ref{lm:dc}.
First of all, let $T$ be the optimal transport map from $\mu_0$ to $\mu_1$ constructed as in \eqref{opte}, via the function $\varphi$ provided by Theorem \ref{bm1}.
Since $\mu_0, \mu_1 \ll \m$, we have $\mu_i = \rho_i \m$ for $i=0,1.$ Furthermore, we have
\begin{equation}
    \mu_t = (T_t)_{\sharp}\mu_0 = \rho_t \m.
\end{equation}
For $\mu_0$-a.e. $x \in \mathcal{M}^{\varphi}$ and $y = T(x)$, 
applying Theorem \ref{maeq}, Theorem \ref{jest} and Lemma \ref{lm:dc}, and using the fact that $y \notin \mathrm{Cut}(x)$, we have
\begin{align*}
    \left(\frac{1}{\rho_t(T_t(x))}\right)^{1/n} &
    =\left(\frac{1}{\rho_0(x)}\frac{\m(X_1(1),...,X_n(1))}{\m(X_1(0),...,X_n(0))}\det(d_xT_1)\right)^{1/n}
\\&\ge \left(\frac{\beta^>_t(x,y)}{\rho_0(x)}\right)^{1/n} + \left(\frac{\det(d_xT_1)}{\rho_0(x)}\frac{\m(X_1(1),...,X_n(1))}{\m(X_1(0),...,X_n(0))}\right)^{1/n}(\beta^<_t(x,y))^{1/n}\\
    &= \left(\frac{\beta^>_t(x,y)}{\rho_0(x)}\right)^{1/n} + \left(\frac{\beta^<_{t}(x,y)}{\rho_1(y)}\right)^{1/n},
\end{align*}
where $X_i(t)$ is a fixed moving frame along the geodesic $T_t(x)$, and the last equality also follows
 directly from the Monge-Amp\`ere equation in Theorem \ref{maeq}.

As for $\mu_0$-a.e. $x \in \mathcal{S}^\varphi$, since $T_t(x) = x$ and $\rho_t(x) = \rho_0(x),$ by Proposition \ref{geodim}, we have
\begin{equation}\label{rhs}
    \beta^>_t(x,x)^{1/n} + \beta^<_t(x,x)^{1/n} \le (1-t)^{k(x)/n} + t^{k(x)/n} \le 1.
\end{equation}
Thus, dividing by $\rho_0(x)^{1/n}$, the interpolation inequality holds trivially on the static set. This concludes the proof.

\subsection{Applications}
As an application, the following Borell-Brascamp-Lieb inequality is a direct result of the interpolation inequality.

\begin{theorem}[Borell-Brascamp-Lieb inequality]\label{thm:bbl}
    Let $(M,\mathcal{D},F)$ be a forward ideal sub-Finslerian manifold equipped with a smooth measure $\m$. Fix $t \in (0,1)$. Let $f,g,h:M\to\mathbb{R}$ be non-negative functions and $A,B\subset M$ be two Borel subsets such that $\int_A f d\m = \int_B g d\m = 1.$ Assume that for all $(x,y)\in (A\times B)\setminus \mathrm{Cut}(M)$ and $z\in Z_t(x,y),$
    \begin{equation}
        \frac{1}{h(z)^{1/n}} \le \left(\frac{\beta^>_t(x,y)}{f(x)}\right)^{1/n} + \left(\frac{\beta^<_t(x,y)}{g(y)}\right)^{1/n},
    \end{equation}
    then we have $\int_M h d\m \ge 1.$
\end{theorem}
\begin{proof}
    We will prove it using a standard argument. Assume first that $A, B$ are bounded Borel sets with compact closures. 
    Set $\mu_0 = f|_A \m$ and $\mu_1 = g|_B \m.$ Letting $T$ be the optimal transport map for 
    $c(x,y)=d_{SF}^2(x,y)$ such that $T_{\sharp}\mu_0 = \mu_1$, and using Theorem \ref{interpolation} alongside the 
    assumption on $h$, we immediately infer that $h(T_t(x)) \ge \rho_t(T_t(x))$ holds on some Borel set 
    $\Omega \subset A$ with full $\mu_0$-measure. Moreover, by Theorem~\ref{ot2} $\mu_t = \rho_t \m$ is absolutely continuous with respect to $\m$. Therefore we have:
    \begin{equation}
        \int_M h(z) d\m \ge \int_{T_t(\Omega)} h d\m \ge \int_{T_t(\Omega)} \rho_t d\m = \mu_t(T_t(\Omega)) = \mu_0(\Omega) = 1.
    \end{equation}

    When $A, B$ are unbounded, we prove it via an approximating argument. 
    Choose bounded subsets $A_{\epsilon} \subset A, B_{\epsilon} \subset B$ such that $\mu_0(A_{\epsilon}) = \mu_1(B_{\epsilon}) = 1-\epsilon$ for a fixed $\epsilon>0.$ Repeating the same process above, we obtain $\int_M h(z) d\m \ge 1-\epsilon.$ Since $\epsilon$ is arbitrary, we conclude the result.
\end{proof}

Let $t\in[0,1]$ and $a,b\ge 0.$ We introduce the $p$-mean for $p\in[-\infty,+\infty].$
When $p\neq 0,\pm \infty,$ it is defined as:
\begin{equation}
    \mathcal{M}^p_t(a,b) = \begin{cases}
        ((1-t)a^p+tb^p)^{1/p}, \quad &\text{if } ab\neq 0 \text{ or }p>0;\\
        0, & \text{if } ab= 0\text{ and }p<0.
    \end{cases}
\end{equation}
The limit cases are defined as follows:
\begin{equation}
    \mathcal{M}^0_t(a,b):=a^{1-t}b^t, \quad \mathcal{M}^{+\infty}_t(a,b):=\max\{a,b\}, \quad \mathcal{M}^{-\infty}_t(a,b):=\min\{a,b\}.
\end{equation}
By convention we also set $\mathcal{M}^p_t(0,\infty):=0$ if $p\in[-\infty,0],$
$\mathcal{M}^p_t(0,\infty):=\infty$ if $p\in (0,\infty],$ and $\mathcal{M}^p_t(\infty,\infty):=\infty$ for $p\in [-\infty,\infty].$
We will use the following H\"older's inequality later. For $a,b,c,d\ge 0,$ and $p^{-1}+q^{-1}=r^{-1}$,
\begin{equation}\label{holder}
    \mathcal{M}^r_t(ac,bd) \le \mathcal{M}^p_t(a,b)\mathcal{M}^q_t(c,d).
\end{equation}

\begin{corollary}[$p$-mean inequality]\label{pmean0}
    Let $(M,\mathcal{D},F)$ be a forward ideal sub-Finslerian manifold equipped with a smooth measure $\m$ and $f,g,h:M\to\mathbb{R}$ be non-negative functions. Fix $t \in (0,1)$ and $p\ge -\frac{1}{n}$, and let $A,B \subset M$ be two non-empty Borel sets such that $\int_A f d\m = \|f\|_{1}$ and $\int_B g d\m = \|g\|_{1}$. If for every $(x,y)\in (A\times B)\setminus \mathrm{Cut}(M)$ and $z\in Z_t(x,y),$
    \begin{equation}
        h(z) \ge \mathcal{M}^p_t\left(\frac{(1-t)^n f(x)}{\beta^>_t(x,y)}, \frac{t^n g(y)}{\beta^<_t(x,y)}\right),
    \end{equation}
    then we have 
    \begin{equation}\label{pmean}
            \int_M h d\m \ge \mathcal{M}^{p/(1+np)}_t\left(\int_M f d\m, \int_M g d\m\right),
    \end{equation}
    where, by convention, $\frac{p}{1+np}=1/n$ when $p=+\infty$, and 
$\frac{p}{1+np}=-\infty$ when $p=-1/n$.
\end{corollary}
\begin{proof}
    For $\|f\|_1\|g\|_1 > 0$, set 
    $$ \hat{h} = \mathcal{M}^{p/(1+np)}_t(\|f\|_{1},\|g\|_{1})^{-1} h = \left((1-t)\|f\|^{\frac{p}{1+np}}_1+t\|g\|_1^{\frac{p}{1+np}}\right)^{-\frac{1+np}{p}} h. $$
    By H\"older's inequality \eqref{holder}, we have for all $z\in Z_t(x,y)$,
    \begin{equation}
        \begin{aligned}
                \hat{h}(z) &\ge \mathcal{M}^{p/(1+np)}_t(\|f\|_1,\|g\|_1)^{-1}\mathcal{M}^p_t\left(\frac{(1-t)^n f(x)}{\beta^>_t(x,y)}, \frac{t^n g(y)}{\beta^<_t(x,y)}\right)\\
                &\ge \mathcal{M}^{1/n}_t\left(\frac{\|f\|_1\beta^>_t(x,y)}{f(x)(1-t)^n}, \frac{\|g\|_1\beta^<_t(x,y)}{g(y)t^n}\right)^{-1}.
        \end{aligned}
    \end{equation}
    Hence $(f/\|f\|_1, g/\|g\|_1, \hat{h})$ satisfies the condition of the Borell-Brascamp-Lieb inequality (Theorem \ref{thm:bbl}), yielding that 
    \begin{equation}
        \int_M \hat{h} d\m \ge 1 \implies \int_M h d\m \ge \mathcal{M}_t^{p/(1+np)}\left(\int_M f d\m, \int_M g d\m\right).
    \end{equation}

    In the case where $\|f\|_1\|g\|_1 = 0$, it suffices to consider the subcase $\|f\|_1 = 0$, $\|g\|_1 > 0$, and $p > 0$; in all other subcases, one can check directly that $\mathcal{M}^{p/(1+np)}(\|f\|_1,\|g\|_1)=0$, so there is nothing to prove.
    Since $A$ is assumed non-empty, we can fix an arbitrary $x \in A$ and let 
    $\psi: B\setminus \mathrm{Cut}(x) \to Z_t(x,B)$ be given by $\psi(y) = \gamma(t)$, where $\gamma$ 
    is the unique geodesic joining $x$ and $y.$ Notice that, since $M$ is forward ideal, $\mathrm{Cut}(x)$ has $\m$-measure zero. By the definition of the forward distortion coefficient, 
    the Jacobian of this map is exactly $\beta^<_t(x,y)$. Therefore, changing variables yields:
    \begin{align*}
        \int_M h d\m &\ge \int_{\psi(B\setminus\mathrm{Cut}(x))} h(z) d\m(z) \ge \int_{B\setminus\mathrm{Cut}(x)} h(\psi(y)) \beta^<_t(x,y) d\m(y)\\
        &\ge \int_B \frac{t^{(1+np)/p} g(y)}{\beta^<_t(x,y)} \beta^<_t(x,y) d\m(y) = t^{(1+np)/p} \|g\|_1 = \mathcal{M}^{p/(1+np)}_t(0, \|g\|_1).
    \end{align*}
     
    The case where $\|f\|_1 + \|g\|_1 = \infty$ can be handled with a standard approximation argument, thus concluding the proof.
\end{proof}

\begin{remark}
    Setting $p=-1/n$ in \eqref{pmean} recovers the Borell-Brascamp-Lieb inequality. For the limit case $p=0$, under the assumption
    \begin{equation}
        h(z) \ge \left(\frac{(1-t)^nf(x)}{\beta^>_t(x,y)}\right)^{1-t}\left(\frac{t^ng(y)}{\beta^<_t(x,y)}\right)^t,
    \end{equation}
    Corollary \ref{pmean0} yields the so-called Pr\'ekopa-Leindler inequality.
\end{remark}

\begin{theorem}[Brunn-Minkowski inequality]\label{BMineq}
    Let $(M,\mathcal{D},F)$ be a forward ideal sub-Finslerian manifold equipped with a smooth measure $\m$. 
    For any non-empty measurable sets $A,B\subset M,$ we have
    $$ \m(Z_t(A,B))^{1/n} \ge \beta^>_t(A,B)^{1/n}\m(A)^{1/n} + \beta^<_t(A,B)^{1/n}\m(B)^{1/n}, $$
    where $\beta^>_t(A,B) = \inf_{(x,y) \in A \times B} \beta^>_t(x,y)$ and 
    $\beta^<_t(A,B) = \inf_{(x,y) \in A \times B} \beta^<_t(x,y)$.
\end{theorem}
\begin{proof}
    Suppose $Z_t(A,B)$ is measurable and put $f = \frac{\beta^>_t(A,B)}{(1-t)^n}\chi_A$, $g = \frac{\beta^<_t(A,B)}{t^n}\chi_B$, and $h = \chi_{Z_t(A,B)},$ where $\chi_S$ denotes the indicator function of the set $S.$
    Then for $(x,y)\in A\times B$ and $z\in Z_t(x,y),$
    \begin{equation}
        1 = h(z) \ge \max\left\{\frac{(1-t)^n f(x)}{\beta^>_t(x,y)}, \frac{t^n g(y)}{\beta^<_t(x,y)}\right\} = \max\left\{\frac{\beta^>_t(A,B)}{\beta^>_t(x,y)}, \frac{\beta^<_t(A,B)}{\beta^<_t(x,y)}\right\}.
    \end{equation}
    Hence, the assumption in Corollary \ref{pmean0} is satisfied for $p=\infty.$ Consequently,
    \begin{equation}
      \begin{aligned}
      \int_M \chi_{Z_t(A,B)} d\m = \m(Z_t(A,B)) &\ge \mathcal{M}^{1/n}_t\left(\int_M f d\m, \int_M g d\m\right)\\
      &= \left(\beta^>_t(A,B)^{1/n}\m(A)^{1/n} + \beta^<_t(A,B)^{1/n}\m(B)^{1/n}\right)^n.
      \end{aligned}
    \end{equation}

    In the case that $Z_t(A,B)$ is not measurable, we can replace it with a measurable set containing it with 
    the same outer measure, and the conclusion follows by the same argument.
\end{proof}
\subsection{Measure contraction properties}
As we mentioned in the introduction, if the distortion coefficients $\beta^<_t(x,y)$ and 
$\beta^>_t(x,y)$ are bounded below by the model coefficients \eqref{eq:model}, Theorem \ref{BMineq} implies the classical Brunn-Minkowski inequality
 $\mathrm{BM}(K,N)$. 
Here we consider the power function lower bound, that is, for some $N\ge n$, for all 
$t\in[0,1]$ 
and $(x,y)\not\in \mathrm{Cut}(M)$, 
$\beta^<_t(x,y)\ge t^N$ and $\beta^>_t(x,y)\ge (1-t)^N.$ 
This condition implies what is known as the measure contraction property $\mathrm{MCP}(0,N).$ 
As a weaker condition than the curvature-dimension condition, this concept was first introduced
 by Ohta and Sturm (\cite{Oh2,St1}). 

The measure contraction property has been extensively studied in sub-Riemannian geometry, 
and several examples satisfying the MCP have been identified, such as Carnot groups \cite{Jui2,Rizzi2}
and the Grushin plane \cite{BR}. As for the sub-Finslerian case, the MCP has been primarily investigated for 
the sub-Finslerian Heisenberg group (see \cite{BT23, BMRT24, BMRT1,MR}), which we will discuss further in 
the next section. 

The following theorem provides the exact equivalence between the MCP and the Brunn-Minkowski inequality we obtained above.

\begin{theorem}
    Let $(M,\mathcal{D},F)$ be an $n$-dimensional forward ideal sub-Finslerian manifold equipped with a smooth measure $\m.$ Let $N\ge n.$ Then the following properties are equivalent:
    \begin{itemize}
        \item[(i)] $\beta^<_t(x,y)\ge t^N$ and $\beta^>_t(x,y)\ge (1-t)^N$, for all $(x,y)\not\in \mathrm{Cut}(M)$ and $t\in[0,1];$
        \item[(ii)] the following Brunn-Minkowski inequality holds: for all non-empty Borel sets $A,B \subset M$ and $\forall t\in[0,1],$
        $$ \m(Z_t(A,B))^{1/n}\ge (1-t)^{N/n}\m(A)^{1/n}+t^{N/n}\m(B)^{1/n}; $$
        \item[(iii)] the measure contraction property $\mathrm{MCP}(0,N)$ is satisfied: for all non-empty Borel sets $A,B \subset M$ and points $x,y\in M,$
        $$ \m(Z_t(x,B))\ge t^N\m(B) \quad \text{and} \quad \m(Z_t(A,y))\ge (1-t)^N\m(A),\quad \forall t\in[0,1]. $$
    \end{itemize}
\end{theorem}
\begin{proof}
    The proof is direct. Plugging (i) into Theorem \ref{BMineq}, we immediately get (ii). 
    
    To see that (ii) implies (iii), we fix a point $x \in M$ and take a sequence of shrinking neighborhoods $A_r \to \{x\}$ as $r \to 0$. Since $\m(A_r) \to 0$, substituting $A_r$ into (ii) yields the forward contraction $\m(Z_t(x,B))\ge t^N\m(B)$. The backward contraction $\m(Z_t(A,y))\ge (1-t)^N\m(A)$ is obtained symmetrically by shrinking $B$ to $\{y\}$. 
    
    Finally, directly from the definition of the forward and backward distortion coefficients in \eqref{defdis}, (iii) implies (i), which completes the proof.
\end{proof}

\section{Randers sub-Finslerian Heisenberg groups}

This section is devoted to investigating a concrete example, sub-Finslerian Heisenberg group.
We define the left-invariant vector fields 
\begin{equation}
    X:=\partial_x-\frac{1}{2}y\partial_z,\quad Y:=\partial_y+\frac{1}{2}x\partial_z,\quad Z:=\partial_z,
\end{equation}
and let $\mathcal{D}:=\mathrm{span}\{X,Y\}$ denote the distribution on $\mathbb{R}^3$.
Endowed with the Randers type norm 
$F_{g,\beta}$ introduced in \eqref{rnorm}, $(\mathbb{H},\mathcal{D},F_{g,\beta})$ is a Randers sub-Finslerian Heisenberg group. Moreover, one can check that this $(\mathbb{H},\mathcal{D},F_{g,\beta})$ is forward ideal.

The geometric meaning of a Randers sub-Finslerian metric (as well as a Randers Finslerian metric)
can be elegantly illustrated by Zermelo's navigation problem. In this interpretation, the metric measures the travel time between two points, and the drift term $\beta$ represents
the ``wind'' which assists or opposes the motion. For computational simplicity, we restrict our attention to a constant drift term, 
\begin{equation}
    F_a(u)=\sqrt{u_1^2+u_2^2}+au_1,\quad \text{for } u=(u_1,u_2).
\end{equation}
Throughout this section, 
we fix $\m$ to be a smooth Lebesgue measure on $\mathbb{H}$.

As we remarked previously, Heisenberg groups equipped with general metric structures have been studied extensively. One interesting fact pointed out in \cite{BMRT1} is that the \emph{curvature exponent} $N$, which is defined as the smallest number for which the measure contraction property $\mathrm{MCP}(0,N)$ is satisfied, is at least $5$, with $N=5$ if and only if the structure is sub-Riemannian.
In this section, through an explicit computation for distortion coefficients of $(\mathbb{H},\mathcal{D}, F_a,\m)$, we establish their measure contraction property, echoing the research of Borza, Magnabosco, Rossi, and Tashiro \cite{BMRT1}.

The dual norm $F_a^*$ of $F_a$ on $\mathbb{R}^2$ is given by 
\begin{equation}
    F_a^*(\xi)=\frac{\sqrt{\xi_1^2+(1-a^2)\xi_2^2}-a\xi_1}{1-a^2}
\end{equation}
for $\xi=(\xi_1,\xi_2)\in (\mathbb{R}^2)^* \cong \mathbb{R}^2$  (see, for instance, \cite[Section 1.1.2]{BRS}). Then,
by Proposition~\ref{prop2.5} the Hamiltonian $H:T^*\mathbb{H}\to\mathbb{R}$
is given by 
\begin{equation}\label{eq3}
    H(\lambda):=\frac{1}{2}r^2=\frac{1}{2}\left(\frac{\sqrt{h_1^2+(1-a^2)h_2^2}-ah_1}{1-a^2}\right)^2,
\end{equation}
where $\lambda \in T^*\mathbb{H}$ is a covector, $r$ is a function of $\lambda$, and $h_1=\langle\lambda, X\rangle, h_2=\langle\lambda,Y\rangle$.
Equation \eqref{eq3} implies that 
\begin{equation}\label{eq4}
    (h_1-ar)^2+h_2^2=r^2.
\end{equation}
Via trigonometric substitution, we introduce the parameterization
\begin{equation}\label{ts}
    h_1=ar+r\cos\psi,\quad
    h_2=r\sin\psi,
\end{equation}
where the function $\psi(t)$ represents the momentum angle at time $t \ge 0$. Next, we aim to 
clarify how $\psi$ evolves over time. In the sub-Riemannian case, $\psi$ is linear with 
respect to $t$, but in the Randers case, the drift alters this behavior.
Since the Hamiltonian $H=\frac{1}{2}r^2$ is conserved along the Hamiltonian flow, we have $\dot{r} = \{H, r\} = 0$. 
Differentiating both sides of \eqref{eq4} with respect to $h_1$ and $h_2$ respectively, we obtain 
\begin{equation}
    \frac{\partial r}{\partial h_1}=\frac{\cos\psi}{1+a\cos\psi},\quad \frac{\partial r}{\partial h_2}=\frac{\sin\psi}{1+a\cos\psi}.
\end{equation}
Differentiating \eqref{ts} with respect to time, we have 
\begin{equation}\label{eq6}
    \dot{h}_1=-r\sin\psi\cdot\dot{\psi},\quad \dot{h}_2=r\cos\psi\cdot\dot{\psi}.
\end{equation}
On the other hand, the Pontryagin Maximum Principle along with the relation $[X,Y]=Z$ implies $\{h_1, h_2\} = \langle \lambda, [X,Y] \rangle = h_3$, yielding:
\begin{equation}\label{eq5}
    \begin{cases}
        \dot{h}_1=\{H,h_1\} = \frac{\partial H}{\partial h_2}\{h_2, h_1\} = -r h_3 \frac{\partial r}{\partial h_2};\\
        \dot{h}_2=\{H,h_2\} = \frac{\partial H}{\partial h_1}\{h_1, h_2\} = r h_3 \frac{\partial r}{\partial h_1};\\
        \dot{h}_3=0.
    \end{cases}
\end{equation}
Hence we set $\theta=h_3(0)$ for some constant $\theta \in\mathbb{R}$, and by combining \eqref{eq6} and \eqref{eq5}, we derive 
\begin{equation}
    \dot{\psi}=\frac{\theta}{1+a\cos\psi},
\end{equation}
which implies that $\psi(t)$ is a non-linear function of time due to the drift term, but it remains strictly monotonic with respect to $t$ (provided $\theta \neq 0$).

In local coordinates, we consider a geodesic $\gamma(t)=(x(t),y(t),z(t))$ on $\mathbb{H}$. The velocity of $\gamma$
satisfies
\begin{equation}
    \begin{cases}
        \dot{x}=\frac{\partial H}{\partial h_1}=\frac{r\cos\psi}{1+a\cos\psi};\\
        \dot{y}=\frac{\partial H}{\partial h_2}=\frac{r\sin\psi}{1+a\cos\psi};\\
        \dot{z}=\frac{1}{2}(x\dot{y}-y\dot{x}).
    \end{cases}
\end{equation}
Integrating over time $t$ via the substitution $dt = \frac{1+a\cos\psi}{\theta} d\psi$, and setting $\psi_0=\psi(0)$ as the initial 
angle, we obtain
\begin{align*}
    x(t)&=\int^{\psi(t)}_{\psi_0}\frac{r\cos\psi}{1+a\cos\psi}\left(\frac{1+a\cos\psi}{\theta}\right)d\psi=\frac{r}{\theta}(\sin\psi(t)-\sin\psi_0);\\
    y(t)&=\int^{\psi(t)}_{\psi_0}\frac{r\sin\psi}{1+a\cos\psi}\left(\frac{1+a\cos\psi}{\theta}\right)d\psi=\frac{r}{\theta}(\cos\psi_0-\cos\psi(t)).
\end{align*}
A similar computation shows that 
\begin{equation}
    z(t)=\int^t_0\frac{1}{2}(x\dot{y}-y\dot{x})d\tau=\frac{r^2}{2\theta^2}[\psi(t)-\psi_0-\sin(\psi(t)-\psi_0)].
\end{equation}
Then the exponential map from the origin is given by $\exp^t_0(r,\theta,\psi_0)=(x(t),y(t),z(t))$. In the case  
where $\theta=0$, the expressions for $x(t), y(t), z(t)$ should be understood by taking the limit $\theta\to 0$. 
Similarly, the reverse exponential map $\overline{\exp}^t_0(r,\theta,\psi_0)=(x',y',z')$ associated with the reverse Hamiltonian given by $\bar{H}(\lambda):=H(-\lambda)$
takes the same expressions as $(x(t),y(t),z(t))$, except that we replace $\psi(t)$ with $\psi'(t)$ satisfying 
\begin{equation}
    \dot{\psi}'=\frac{\theta}{1-a\cos\psi'}.
\end{equation}
Here $\psi'_0:=\psi'(0)$, and $\psi'(t)$ denotes the momentum angle of the reverse geodesic.

To maintain the continuous flow of the main narrative, the lengthy but direct computation of the 
Jacobian determinant for the exponential map is deferred to Appendix C. The explicit result is given as 
follows.

\begin{lemma}\label{dcr}
    The exponential map of the Randers sub-Finslerian Heisenberg group $(\mathbb{H},\mathcal{D}, F_a,\m)$ satisfies 
    \begin{equation}
        \det(D_{(r,\theta,\psi_0)}\exp^t_0)=\frac{4r^3t}{\theta^4(1+a\cos\psi(t))}\sin\left(\frac{\psi(t)-\psi_0}{2}\right)\left(\frac{\psi(t)-\psi_0}{2}\cos\left(\frac{\psi(t)-\psi_0}{2}\right)-\sin\left(\frac{\psi(t)-\psi_0}{2}\right)\right),
    \end{equation}
    admitting a smooth extension at $\theta=0$, given by 
    \begin{equation}
        \det(D_{(r,0,\psi_0)}\exp^t_0)=-\frac{r^3t^5}{12(1+a\cos\psi_0)^5}.
    \end{equation}
\end{lemma}
Via Lemma~\ref{lm:dc}, the expression for the distortion coefficients, we derive the explicit 
expression for the distortion coefficients of the Randers sub-Finslerian Heisenberg group.
\begin{lemma}[Randers sub-Finslerian Heisenberg distortion coefficients]\label{rsfdis}
    Let $(\mathbb{H},\mathcal{D}, F_a,\m)$ be a Randers sub-Finslerian Heisenberg group and $q\not\in\mathrm{Cut}(0)$. Then 
    \begin{align*}
            \beta_t^<(0,q)&=t\frac{1+a\cos\psi(1)}{1+a\cos\psi(t)}\frac{\sin(\frac{\psi(t)-\psi_0}{2})\left(\frac{\psi(t)-\psi_0}{2}\cos(\frac{\psi(t)-\psi_0}{2})-\sin(\frac{\psi(t)-\psi_0}{2})\right)}{\sin(\frac{\psi(1)-\psi_0}{2})\left(\frac{\psi(1)-\psi_0}{2}\cos(\frac{\psi(1)-\psi_0}{2})-\sin(\frac{\psi(1)-\psi_0}{2})\right)}\quad \forall t\in[0,1];\\
            \beta^>_t(0,q)&=(1-t)\frac{1-a\cos\psi'(1)}{1-a\cos\psi'(1-t)}\frac{\sin(\frac{\psi'(1-t)-\psi'_0}{2})\left(\frac{\psi'(1-t)-\psi'_0}{2}\cos(\frac{\psi'(1-t)-\psi'_0}{2})-\sin(\frac{\psi'(1-t)-\psi'_0}{2})\right)}{\sin(\frac{\psi'(1)-\psi'_0}{2})\left(\frac{\psi'(1)-\psi'_0}{2}\cos(\frac{\psi'(1)-\psi'_0}{2})-\sin(\frac{\psi'(1)-\psi'_0}{2})\right)}\quad \forall t\in[0,1],
    \end{align*}
    where $\psi_0, \psi(1)$ only depend on the initial covector of the unique geodesic joining $0$ and $q$.
\end{lemma}
By Lemma~\ref{dcr}, for $q\notin \mathrm{Cut}(0)$, one has $0 < \left|\frac{\psi(1)-\psi_0}{2}\right| < \pi$. Therefore
the formulas in Lemma~\ref{rsfdis} should be understood in this domain, in which case $\beta^<_t(0,q), \beta^>_t(0,q) > 0$ for all $t\in(0,1)$.

When $a=0$, the drift term vanishes and the structure reduces to the sub-Riemannian 
Heisenberg group. In such a case, the distortion coefficient is given by 
\begin{equation}
    \beta_t(0,q):=\beta_t^<(0,q)=\beta_{1-t}^>(0,q)=t\frac{\sin(\frac{t\theta}{2})}{\sin(\frac{\theta}{2})}\frac{\sin(\frac{t\theta}{2})-\frac{t\theta}{2}\cos(\frac{t\theta}{2})}{\sin(\frac{\theta}{2})-\frac{\theta}{2}\cos(\frac{\theta}{2})},\quad \forall t\in[0,1].
\end{equation}
The sharp bound for the Heisenberg distortion (see \cite[Lemma 50]{BR}) claims that if $q\not\in\mathrm{Cut}(0)$,
\begin{equation}\label{SRdc}
    \beta_t(0,q)\ge t^N,\quad \forall t\in[0,1], \text{ with } N\ge 5.
\end{equation}
By left invariance, this holds for any pair of points $(q_0,q)\in\mathbb{H}\times\mathbb{H}$ 
with $q\not\in \mathrm{Cut}(q_0)$. Therefore, sub-Riemannian Heisenberg groups satisfy 
$\mathrm{MCP}(0,5)$. In the present paper, we show that for any $a$ with $0 < |a| < 1$, Randers 
sub-Finslerian Heisenberg groups satisfy $\mathrm{MCP}(0,N)$ for some $N=N(a)>5$, 
thereby providing a concrete special case of the general results of 
Borza--Magnabosco--Rossi--Tashiro \cite{BMRT1}.
To this end, we prove the following key estimate.
 
\begin{proposition}\label{611}
Let $(\mathbb{H},\mathcal{D}, F_a,\m)$ be a Randers sub-Finslerian Heisenberg group with $a\in (-1,1)$. Then there exists an exponent $N>5$, depending only on $a$, such that for all points $q_0,q\in\mathbb{H}$ with $q\notin \mathrm{Cut}(q_0)$, the following estimates hold for all $t\in[0,1]$:
\begin{align*}
    \beta_t^<(q_0,q) &\ge t^N,\\
    \beta_t^>(q_0,q) &\ge (1-t)^N.
\end{align*}
In fact, one may choose
\[
N=\frac{1+|a|}{1-|a|}\left(\frac{2\pi|a|}{1-|a|}+4\right)+1,
\]
which is strictly greater than $5$ when $a\neq 0$.
\end{proposition}
\begin{proof}
We only prove the estimate for $\beta_t^<(0,q)$, as the proof for
$\beta_t^>(0,q)$ is analogous. Let
$\gamma(t)=\exp_0^t(r,\theta,\psi_0)$
be the unique geodesic joining $0$ to $q$. When $\theta=0$, Lemma~\ref{dcr}
together with Lemma~\ref{rsfdis} yields
$\beta_t^<(0,q)=t^5$.
Hence, it remains to consider the case $\theta\neq 0$. Without loss of
generality, we assume $\theta>0$; the case $\theta<0$ follows by the same
argument. In this case,
 $\psi(t)$ is strictly increasing in $t$, 
and therefore $
\frac{\psi(1)-\psi_0}{2}\in(0,\pi).$

    We begin by introducing the following convenient notation.
    Since $\dot{\psi}=\frac{\theta}{1+a\cos\psi}$, integrating it yields 
    \begin{equation}
        t=\frac{1}{\theta}\int^{\psi(t)}_{\psi_0} (1+a\cos\psi) d\psi=\frac{1}{\theta}(\psi(t)-\psi_0+a(\sin\psi(t)-\sin\psi_0)).
    \end{equation}
    Set $x=\frac{\psi(t)-\psi_0}{2}$ and let 
    \begin{equation}
        I_a(x)=2x+a(\sin(2x+\psi_0)-\sin\psi_0).
    \end{equation}
    Then $t=\frac{I_a(x)}{I_a(x_1)}$ where $x_1 = \frac{\psi(1)-\psi_0}{2}$, and 
    \begin{equation}
        \beta^<_t(0,q)=t\frac{1+a\cos\psi(1)}{1+a\cos\psi(t)}\frac{\sin x}{\sin x_1}\left(\frac{\sin x-x\cos x}{\sin x_1-x_1\cos x_1}\right).
    \end{equation} 
    To establish our result, we construct the following function 
    \begin{equation}
        f(x)=\frac{g(x)}{I_a^s(x)h(x)},
    \end{equation}
    where $g(x)=\sin x(\sin x-x\cos x)$, $h(x)=1+a\cos(2x+\psi_0)$, and $s\in\mathbb{R}$ is a parameter to be determined later. Observing that $\beta^<_t(0,q)\ge t^{s+1}$ for all $q\notin \Cut{0}$ and all $t\in [0,1]$ is equivalent to $f$
    being non-increasing on $(0,\pi)$, we analyze the sign of its logarithmic derivative.
    Taking the logarithm and computing the derivative, we obtain 
    \begin{align}\label{eq61}
    \frac{f'(x)}{f(x)} 
    &= \frac{g'(x)}{g(x)} - s\frac{I'_a(x)}{I_a(x)} - \frac{h'(x)}{h(x)} \notag\\
    &= \frac{\tfrac{1}{2}\sin(2x)-x\cos(2x)}{\sin^2 x-\tfrac{1}{2}x\sin(2x)}
     - s\frac{2+2a\cos(2x+\psi_0)}{2x+a(\sin(2x+\psi_0)-\sin\psi_0)}
     + \frac{2a\sin(2x+\psi_0)}{1+a\cos(2x+\psi_0)}.
    \end{align}
    First, we show that $s > 4$ whenever $a \neq 0$.
    As a matter of fact, we fix $a\neq 0$ and assume that $q=\exp^t_0(r,\theta,\psi_0)\in \mathbb{H}\setminus\Cut{0}$ is such that the initial momentum angle $\psi_0$ of the unique geodesic joining $0$ to $q$ additionally satisfies $a\sin\psi_0>0$. Taking the Taylor expansion of $\frac{f'(x)}{f(x)}$ near $x=0,$ we obtain  
    \begin{equation}
        \frac{f'(x)}{f(x)}=\frac{4-s}{x}+\frac{(s+2)a\sin\psi_0}{1+a\cos\psi_0}+O(x).
    \end{equation}
    When $s\le 4$, for $x$ sufficiently close to $0$, one can check that $\frac{f'(x)}{f(x)}>0$. Therefore, $\beta^<_t(0,q)\ge t^{5}$ does not hold for all $q\notin\Cut{0}$ unless $a=0$.

    To determine an explicit sufficient value of $N$, we consider the following estimate.
    Taking the Taylor expansion of $\frac{g'(x)}{g(x)}$ near $x=0$, it can be shown that $\frac{g'(x)}{g(x)}\le \frac{4}{x}$ for $x\in (0,\pi)$.
    Let us prove this fact. Letting $H(x)=xg'(x)-4g(x)$ and taking derivatives, we obtain
    \begin{align*}
        H'(x)&=xg''(x)-3g'(x)=2x^2\sin(2x)-\frac{3}{2}\sin(2x)+3x\cos(2x);\\
        H''(x)&=xg'''(x)-2g''(x)=2x(2x\cos(2x)-\sin(2x)).
    \end{align*}
    There exists a unique $x_0\in(\pi/2,\pi)$ such that
    \[
    2x_0\cos(2x_0)-\sin(2x_0)=0.
    \]
    Indeed, the function $2x\cos(2x)-\sin(2x)$ is strictly decreasing on $(0,\pi/2)$ and strictly increasing on $(\pi/2,\pi)$, and it changes its sign across $(\pi/2,\pi)$. Therefore, $H''(x)\le 0$ on $(0,x_0)$ and $H''(x)\ge 0$ on $(x_0,\pi)$. Hence $H'$ is decreasing on $(0,x_0)$ and increasing on $(x_0,\pi)$, so the minimum of $H'$ on $[0,\pi]$ is attained at $x_0$.
    Since
    \[
    H'(0)=0,\qquad H'(\pi)=3\pi>0,
    \]
    $H'$ has exactly one zero in $(x_0,\pi)$. Thus $H$ decreases first and then increases on $(0,\pi)$, so its maximum on $[0,\pi]$ is attained at one of the endpoints. Since
    \begin{equation}
        H(0)=0,\qquad
        H(\pi)=-\pi^2<0,
    \end{equation}
    we conclude that $H(x)\le 0$, which yields the estimate $\frac{g'(x)}{g(x)}\le \frac{4}{x}$ for $x\in (0,\pi)$.

    It remains to bound the last two terms in equation \eqref{eq61}.
    By the mean value theorem, we have $|\sin(2x+\psi_0)-\sin\psi_0|\le 2x$. Therefore, $I_a(x) \le 2x(1+|a|)$.
    Additionally, $I_a'(x) \ge 2(1-|a|)$. Thus, we have 
    \begin{equation}
        -s\frac{I'_a(x)}{I_a(x)} \le -\frac{s}{x}\left(\frac{1-|a|}{1+|a|}\right) \quad \text{for } x\in (0,\pi).
    \end{equation}
    Furthermore, bounding the last term yields $\frac{2a\sin(2x+\psi_0)}{1+a\cos(2x+\psi_0)} \le \frac{2|a|}{1-|a|}$.
    Combining these with $\frac{g'(x)}{g(x)}\le \frac{4}{x}$, we derive an upper bound for the logarithmic derivative:
    \begin{equation}
        \frac{f'(x)}{f(x)} \le \frac{4}{x} - \frac{s}{x}\left(\frac{1-|a|}{1+|a|}\right) + \frac{2|a|}{1-|a|} = \frac{1}{x}\left(4 - s\left(\frac{1-|a|}{1+|a|}\right) + \frac{2|a|x}{1-|a|}\right).
    \end{equation}
    Hence if we take  
    \begin{equation}
        s \ge \frac{1+|a|}{1-|a|}\left(\frac{2\pi|a|}{1-|a|}+4\right),
    \end{equation}
    then $\frac{f'(x)}{f(x)}$ is non-positive for all $x \in (0,\pi)$, and consequently $f(x)$ is a decreasing function. Since $x \le x_1$, we have $f(x) \ge f(x_1)$. 
    Expanding this inequality and using the fact that $t = \frac{I_a(x)}{I_a(x_1)}$, we obtain $\beta^<_t(0,q) \ge t^{N}$ with $N=s+1$, completing our proof.
\end{proof}

The following results follow immediately from Theorem~\ref{interpolation} and Theorem~\ref{BMineq}, together with Proposition~\ref{611}.
\begin{corollary}
    Let $(\mathbb{H},F_a,\m)$ be a Randers sub-Finslerian Heisenberg group. Then for a given $a\in (-1,1)$ and any $\mu_0\in\mathcal{P}^{ac}_c(\mathbb{H})$ and $\mu_1\in \mathcal{P}^{ac}_c(\mathbb{H})$, letting $\mu_t=(T_t)_{\sharp}\mu_0=\rho_t m$
    be the unique Wasserstein geodesic connecting $\mu_0$ with $\mu_1$, we have 
    \begin{equation}
        \frac{1}{\rho_t(T_t(x))^{1/3}}\ge \frac{(1-t)^{N/3}}{\rho_0(x)^{1/3}}+\frac{t^{N/3}}{\rho_1(T_1(x))^{1/3}}, \quad \mu_0\text{-}\mathrm{a.e.}\; \forall t\in [0,1],
    \end{equation}
    where $N \ge \frac{1+|a|}{1-|a|}\left(\frac{2\pi|a|}{1-|a|}+4\right)+1$.
\end{corollary}
\begin{corollary}
    Let $(\mathbb{H},F_a,\m)$ be a Randers sub-Finslerian Heisenberg group. For a given $a\in (-1,1)$ and all non-empty Borel sets $A,B\subset \mathbb{H}$, we have 
    \begin{equation}
        m(Z_t(A,B))^{1/3}\ge (1-t)^{N/3}m(A)^{1/3}+t^{N/3}m(B)^{1/3},\quad \forall t\in [0,1],
    \end{equation}
    where $N \ge \frac{1+|a|}{1-|a|}\left(\frac{2\pi|a|}{1-|a|}+4\right)+1$.
    Hence for such $a$, $(\mathbb{H},F_a,m)$ satisfies $\mathrm{MCP}(0,N)$.
\end{corollary}

\begin{remark}
    The estimate obtained in Proposition~\ref{611} is generally not sharp. 
Therefore, the quantity 
\begin{equation}
\frac{1+|a|}{1-|a|}\left(\frac{2\pi|a|}{1-|a|}+4\right)+1
\end{equation}
should be regarded as an explicit upper bound for the curvature 
exponent of $(\mathbb{H},F_a,\mathfrak{m}).$ Although we do not 
derive an explicit formula for the curvature exponent itself, the 
proof of Proposition~\ref{611} shows that determining it is equivalent 
to finding the smallest value of $s$ such that the logarithmic 
derivative in~\eqref{eq61} is non-positive for every 
$x\in (0,\pi).$ The corresponding curvature exponent is then 
given by $N=s+1.$ In particular, $N$ depends continuously on 
the parameter $a$, which is consistent with 
\cite[Theorem~4.7]{BMRT1}, where the continuity of the curvature 
exponent with respect to sub-Finslerian metrics under the 
so-called $C^k$-strong topology (introduced in 
\cite[Definition~4.1]{BMRT1}) was established.

\end{remark}

\section{Further discussions}
As mentioned earlier, classical curvature-dimension conditions generally fail to hold in sub-Finslerian geometry. Hence, developing a new framework for synthetic curvature bounds is highly desirable. In this concluding section, we briefly highlight two recent approaches that have made significant progress in this direction.

The first unifying framework was proposed by Barilari, Mondino, and Rizzi \cite{BMR}. Building upon an idea introduced by Villani in \cite[Remark 14.23]{Vi}, they replace the classical curvature parameter $K$ with a generalized distortion coefficient $\beta$, which plays the role of curvature, thereby establishing a unified synthetic curvature theory. Our analysis of the Randers sub-Finslerian Heisenberg group provides a concrete class of non-sub-Riemannian examples that naturally fit into their theoretical framework.

Another notable approach is the Quasi Curvature-Dimension (QCD) condition introduced by Milman \cite{Mil}, which is motivated by interpolation inequalities established on ideal sub-Riemannian manifolds. The QCD condition is a relaxation of the classical CD condition, introducing a ``slack'' coefficient $Q \ge 1$ (which reduces to the classical CD condition when $Q=1$). Therefore, it is weaker than CD but stronger than the Measure Contraction Property (MCP). In fact, $\mathrm{MCP}(K,N)$ implies $\mathrm{QCD}(2^{N-n},K,N)$, where $n$ denotes the topological dimension of $M$ \cite[Proposition 2.4]{Mil}. 

Furthermore, Milman utilizes needle decomposition techniques to derive several functional inequalities—such as the Poincar\'e and log-Sobolev inequalities—for a certain class of sub-Riemannian manifolds satisfying the QCD condition. Thanks to the measure contraction property established for Randers sub-Finslerian Heisenberg groups in this paper, we expect that similar results and functional inequalities can be further extended to these non-sub-Riemannian settings.

\appendix
\section{Discussion on Conjugate Points}

This appendix is devoted to discussing conjugate points in sub-Finslerian geometry. As in the usual case, 
the conjugate points are those points where exponential maps fail to have maximal rank. Since there is 
no Chern connection or flag curvature on sub-Finslerian manifolds, the classical index form induced from the second variation 
of the length functional is not available in our case. The index form plays an important role in studying conjugate points in 
Riemannian or Finslerian geometry. Thanks to the pioneering work of Agrachev and Sarychev concerning the theory of conjugate points 
on general control systems (see e.g. \cite[Chapter 8]{Agr},\cite{Sa}), a generalized index form and techniques have been developed and have been shown effective in sub-Riemannian 
geometry. Adapting these, we will extend some well-known conclusions in Riemannian or Finslerian geometry to sub-Finslerian geometry, 
and prove Theorem~\ref{conjugatepoint}. It should be noted that, in this section, 
we do not require the manifold to be forward ideal.

Firstly, we introduce the second variation formula of the energy functional.
\begin{lemma}\label{lmA1}
    Let $\gamma_u:[0,1]\to M$ be a normal geodesic joining $x$ and $y$ that contains no abnormal segments, satisfying 
    \begin{equation}\label{eq:1}
        \langle \lambda, D_uE_{x}(\cdot) \rangle = D_uJ(\cdot)
    \end{equation}
    for some $\lambda\in T^*_yM.$ Then we have 
    \begin{equation}
        \text{Hess}_u(J|_{E^{-1}_{x}(y)})(v)= D^2_uJ(v,v) - \langle\lambda,D^2_uE_{x}(v,v)\rangle,
    \end{equation}
    where $D^2_uE_{x}(v,v)$ is exactly the second variation of the end-point map
     (see the explicit expression in Proposition~\ref{var_endpoint}).
\end{lemma}
\begin{proof}
    Working in a local coordinate chart around $y$, let us choose a smooth one-parameter variation $u_s$ such that 
    $u_0 = u$ and $E_{x}(u_s) \equiv y.$ Differentiating the constraint identity gives
    \begin{equation}\label{2}
        \begin{aligned}
            &D_{u_s}E_{x}(\dot{u}_s)=0;\\
            &D^2_{u_s}E_{x}(\dot{u}_s,\dot{u}_s)+D_{u_s}E_{x}(\ddot{u}_s )=0.
        \end{aligned}
    \end{equation}
    Evaluating the second variation at $s=0$ with $v = \dot{u}_0$, we obtain
    \begin{equation}
        \begin{aligned}
            \text{Hess}_u(J|_{E^{-1}_{x}(y)})(v)&:=\frac{d^2}{ds^2}\Big|_{s=0}J(u_s)\\
            &=D^2_uJ(v,v)+D_uJ(\ddot{u}_0)\\
            &=D^2_uJ(v,v)+\langle \lambda, D_uE_{x}(\ddot{u}_0)\rangle \quad\quad \text{(using formula \eqref{eq:1})}\\
            &=D^2_uJ(v,v)-\langle \lambda, D^2_uE_{x}(v,v) \rangle \quad\quad \text{(using formula \eqref{2})}.
        \end{aligned}
    \end{equation}
    This completes the proof.
\end{proof}

As in Riemannian geometry, the degeneracy of the Hessian of the energy functional implies the appearance of conjugate points 
in sub-Finslerian geometry. To this end, we introduce the following reparametrization. Set $u^s(t):=su(st)$, which is 
the control of $\gamma_{u^s}(t)=\gamma_u(st).$ Then the corresponding Lagrangian multiplier is given by 
$\lambda^s=s(P_{s,1})^*\lambda\in T^*_{\gamma_u(s)}M,$ where $P_{s,t}$ denotes the flow of the non-autonomous horizontal vector 
field $X_u(q)=u^iX_i$, where $\{X_i\}_{i=1}^k$ denotes a local frame for the horizontal distribution $\mathcal{D}$, and $X_u$ 
represents the horizontal vector field associated with the control $u$.

\begin{proposition}
    Let $s\in(0,1]$ and let $\gamma_u:[0,1]\to M$ be a normal geodesic containing no abnormal segments. Then $\gamma_u(s)$ is conjugate to $\gamma_u(0)$ if and only if $\text{Hess}_{u^s}(J|_{E^{-1}_{q_0}(\gamma_u(s))})$ is a degenerate quadratic form, where $q_0=\gamma_u(0).$
\end{proposition}
\begin{proof}
    Without loss of generality, we only prove the case for $s=1$. The general case for $s \in (0,1)$ follows directly from the reparametrization $u^s$. Let $q_0 = \gamma_u(0)$ be the starting point and $x = \gamma_u(1)$ be the endpoint. We choose local coordinates around $x$ such that $\lambda=(\xi,x)$ in $T^*M.$ It is sufficient to consider the exponential map restricted to the critical manifold:
    \begin{equation}
        C=\{(u,\xi,x) \mid D_uJ-\langle \xi, D_uE_{q_0}(\cdot)\rangle =0, E_{q_0}(u)=x\}.
    \end{equation}
    Then the tangent space of $C$ at $(u,\xi,x)$ is given by the linearized equations:
    \begin{equation}
        T_{(u,\xi,x)}C=\left\{(u',\xi',x') \mathrel{}\middle|\mathrel{} \begin{cases}
            D^2_uJ(u',\cdot)-\langle \xi, D^2_uE_{q_0}(u',\cdot)\rangle - \langle \xi', D_uE_{q_0}(\cdot)\rangle =0\\
            D_uE_{q_0}(u')=x'
        \end{cases}\right\}.
    \end{equation}
    We introduce the following linear operators:
    \begin{equation}
        \begin{aligned}
            Q(u')&=D^2_uJ(u',\cdot)-\langle \xi, D^2_uE_{q_0}(u',\cdot)\rangle;\\
            B(\xi')&=\langle \xi', D_uE_{q_0}(\cdot)\rangle.
        \end{aligned}
    \end{equation}
    When $x=\gamma_u(1)$ is conjugate to $q_0$, the differential of the exponential map
    \begin{equation}
        \begin{aligned}
            d\exp:T_{(u,\xi,x)}C &\to T_xM\\
            (u',\xi',x') &\mapsto x'
        \end{aligned}
    \end{equation}
    is degenerate. This implies there exists a non-trivial tangent vector $(u',\xi',0) \in T_{(u,\xi,x)}C$. 
    Since $x'=0$, we immediately have $D_uE_{q_0}(u') = 0$, meaning $u' \in \ker D_uE_{q_0}$. 
    
    Since the geodesic contains no abnormal segments, the differential $D_uE_{q_0}$ is not degenerate. 
    If $u'=0$, the first equation in the tangent space would yield $\langle \xi', D_uE_{q_0}(\cdot)\rangle = 0$, 
    which implies $\xi'=0$ by surjectivity.
     This contradicts the fact that $(u',\xi',0)$ is a non-trivial vector. Thus, we must have $u' \neq 0$.
Furthermore, the first equation of the tangent space yields the operator identity $Q(u') - B(\xi') = 0$. 
    When evaluating this identity on any test vector $v \in \ker D_uE_{q_0}$, the term $B(\xi')(v) = \langle \xi', D_uE_{q_0}(v)\rangle$ vanishes identically. Thus, we obtain $Q(u')(v) = 0$ for all $v \in \ker D_uE_{q_0}$. 
    Since $u' \in \ker D_uE_{q_0}$ is non-zero, this precisely means that the quadratic form $Q$, which indeed coincides with the Hessian of the restricted energy functional by Lemma~\ref{lmA1}, is degenerate on $\ker D_uE_{q_0}$. The opposite implication is analogous, which concludes the proof.
\end{proof}

Next, we will show that conjugate points are isolated provided that the manifold contains no abnormal segments along the geodesic. This good phenomenon indeed arises from the convexity of the Hamiltonian and the absence of abnormal segments.
\begin{proposition}[See \cite{Agr}, Corollary 8.51]\label{prop:1}
    Let $\gamma:[0,1]\to M$ be a normal geodesic that does not contain abnormal segments. Then the conjugate points to $\gamma(0)$ along $\gamma(t)$ are isolated, i.e., the set 
    \begin{equation}
        T_c:=\{t>0 \mid \gamma(t)\text{ is conjugate to }\gamma(0)\},
    \end{equation}
    is discrete.
\end{proposition}

This proposition is a direct corollary of the following theorem.
\begin{theorem}\label{thm:A1}
    Let $\gamma:[0,1]\to M$ be a normal extremal containing no abnormal segments. Assume there exists $t_0>0$ such that $t_i\downarrow t_0$ with $\gamma(t_i)$ conjugate to $\gamma(0)$ along $\gamma$. Then there exists $\varepsilon>0$ such that 
    \begin{itemize}
        \item for all $\tau\in[t_0,t_0+\varepsilon],$ $\gamma(\tau)$ is conjugate to $\gamma(0);$
        \item $\gamma|_{[t_0,t_0+\varepsilon]}$ is an abnormal extremal path.
    \end{itemize}
\end{theorem}

\begin{proof}
    Let $\lambda_t$ be the normal lift of $\gamma,$ and $\mathcal{V}_t$ denote the vertical subspace of $T_{\lambda_t}T^*M.$
    Suppose $\gamma(t_0)$ is conjugate to $\gamma(0),$ which implies 
    $e^{-t_0\vec{H}}_*\mathcal{V}_{t_0}\cap \mathcal{V}_0$ is non-trivial.
    Assume that $\dim(e^{-t_0\vec{H}}_*\mathcal{V}_{t_0}\cap \mathcal{V}_0)=k$ for some integer $k>0.$ Since both $\mathcal{V}_0$ and $e^{-t_0\vec{H}}_*\mathcal{V}_{t_0}$ are Lagrangian subspaces, there exists a Darboux basis $E_1,...,E_n,F_1,...,F_n$ of the symplectic vector space $T_{\lambda_0}(T^*M)$ such that (see \cite{Agr}, Lemma 8.48)
    \begin{align}
        e^{-t_0\vec{H}}_*\mathcal{V}_{t_0}&=\text{span}\{E_1,...,E_k,E_{k+1}+F_{k+1},...,E_n+F_n\};\\
        \mathcal{V}_0&=\text{span}\{E_1,...,E_n\}.
    \end{align}      
    We now work in the local coordinates $(p,x)\in\mathbb{R}^n\times\mathbb{R}^n$ introduced by this Darboux basis. In
     this setting, the subspace defined by the equation $\{x=0\}$ coincides with our vertical space $\mathcal{V}_0$. Meanwhile, we identify its natural Lagrangian complement defined by the equation $\{p=0\}$ as the algebraic horizontal space. By definition, we can express $e^{-t_0\vec{H}}_*\mathcal{V}_{t_0}$ as a graph parameterized by $p$:
    \begin{equation}
        e^{-t_0\vec{H}}_*\mathcal{V}_{t_0}=\left\{(p, x)\in \mathbb{R}^n\times\mathbb{R}^n \mathrel{}\middle|\mathrel{} x=S_{t_0}p, S_{t_0}=\begin{pmatrix} 0_k &0\\0& I_{n-k} \end{pmatrix} \right\}.
    \end{equation}
    Since transversal intersection is an open condition, any $n$-dimensional subspace sufficiently close to $e^{-t_0\vec{H}}_*\mathcal{V}_{t_0}$ 
    remains transversal to the horizontal space $\{p=0\}$. Hence, for $t>t_0$ close to $t_0$, 
    the evolving Lagrangian subspace $e^{-t\vec{H}}_*\mathcal{V}_{t}$ can still be represented 
    as a graph: 
    \begin{equation}
        e^{-t\vec{H}}_*\mathcal{V}_{t}=\{(p,S_tp) \mid p\in \mathbb{R}^n\},
    \end{equation}
    where $S_t$ is a symmetric matrix. In fact, since it is a Lagrangian subspace, the symplectic form 
    vanishes:
    \begin{equation}
        \omega((p_1,S_tp_1),(p_2,S_tp_2))=p^T_1S_tp_2-p_2^TS_tp_1=p_1^T(S_t-S_t^T)p_2=0.
    \end{equation}

    On the other hand, for a fixed $p$, set $\xi_t=(p, S_t p) \in e^{-t\vec{H}}_*\mathcal{V}_{t}$.
     By the nonnegativity of the Hamiltonian, the quadratic form associated with $S_t$ is non-decreasing: 
\begin{equation}
    0\le \omega(\xi_t,\dot{\xi}_t)=\omega((p,S_tp),(0,\dot{S}_tp))=p^T\dot{S}_tp.
\end{equation}
    This fact implies that if there exists a sequence $t_i\downarrow t_0$ conjugate to $\gamma(0),$ then for some $t_1>t_0$, there exists a non-zero vector $\bar p$ such that $S_{t_1}\bar p=0.$ Consequently, we have
    \begin{equation}
        \bar{p}^TS_{t_0}\bar p\le \bar{p}^TS_t\bar p\le\bar{p}^TS_{t_1}\bar{p}=0.
    \end{equation}
    This forces $S_t \bar{p} \equiv 0$ for all $t\in [t_0,t_1].$ Therefore, any $\gamma(t)$ on this interval is conjugate to $\gamma(0).$

    Now we construct an abnormal lift of $\gamma|_{[t_0,t_1]}$. Consider the specific initial tangent 
    vector $\xi_0 = (\bar{p}, 0)^T \in \mathcal{V}_0$. We define the corresponding Jacobi field along 
    the extremal as $J(t) := e^{t\vec{H}}_* \xi_0$, and the vanishing spatial component, i.e. $x(0)=0$,
    makes $J(t)$ a strictly vertical Jacobi field. 
    Recalling that
    \begin{equation}
        \begin{pmatrix} \dot{p} \\ \dot{x} \end{pmatrix} =
        \begin{pmatrix}
            -A &-R\\B & A^T
        \end{pmatrix}
        \begin{pmatrix}
            p(t)\\ 0
        \end{pmatrix},
    \end{equation}
    we obtain $\dot{x}(t) = B(t)p(t) = 0$. Via the explicit expression \eqref{abr}, 
    the kernel of $B(t)$ coincides exactly with the annihilator 
    $\text{Ann}(\mathcal{D}_{\gamma(t)})$. Hence, this $p(t)$ 
    forms an abnormal extremal path lifting $\gamma|_{[t_0,t_1]}$, which completes the proof.
\end{proof}

For a fixed $s \in (0,1]$, we define a norm on the control space:
\begin{equation}
    \|v\|_{s}:=\left(\int^1_0\frac{1}{2}\frac{\partial^2F^2}{\partial u^i\partial u^j}\Big|_{u^s}(v,v)dt\right)^{1/2}.
\end{equation}
To study the degeneracy of the restricted Hessian of the energy functional, we introduce the following notation:
\begin{equation}
    \alpha(s):=\inf\{\|v\|^2_{s}-\langle\lambda^s,D^2_{u^s}E_{q_0}(v,v)\rangle \mid \|v\|^2_s=1, v\in\ker D_{u^s}E_{q_0}\}.
\end{equation}
Obviously, $\alpha(s)$ denotes the minimal eigenvalue of the restricted Hessian $\text{Hess}_{u^s}(J|_{E^{-1}_{q_0}(\gamma_u(s))})$. Conjugate points appear exactly when $\alpha(s)$ equals $0$. We have the following properties.

\begin{lemma}\label{lm:1}
    Assume that a normal extremal trajectory $\gamma_u:[0,1]\to M$ contains no abnormal segments. Set $q_0=\gamma_u(0)$. Then $\alpha(s)$ is continuous on $(0,1]$ and has the following properties:
    \begin{itemize}
        \item[(1)] $\alpha(0):=\lim\limits_{s\to 0^+}\alpha(s)=1$;
        \item[(2)] $\alpha(s)=0$ implies that $\text{Hess}_{u^s}(J|_{E^{-1}_{q_0}(\gamma_u(s))})$ is degenerate;
        \item[(3)] $\alpha(s)$ is monotone decreasing;
        \item[(4)] if $\alpha(\bar s)=0$ for some $\bar s>0,$ then $\alpha(s)<0 $ for all $s>\bar s.$
    \end{itemize}
\end{lemma}
\begin{proof}
    We follow the standard argument in \cite[Lemma 8.55]{Agr}. We construct a bounded compact linear operator $Q_s$ on the control space from the second variation of the endpoint map such that 
    \begin{equation}
        \langle (I-Q_s)v,v\rangle_s:=\|v\|^2_s-\langle\lambda^s, D^2_{u^s}E_{q_0}(v,v)\rangle.
    \end{equation}
    Let us first prove the continuity of $\alpha(s)$. From the explicit formula for the second variation of the endpoint map (c.f. Proposition~\ref{var_endpoint}), the mapping $s\mapsto Q_s$ is continuous in the operator norm topology. On the other hand, since there are no abnormal segments, the differential $D_{u^s}E_{q_0}$ is surjective. Therefore, its kernel $\ker D_{u^s}E_{q_0}$ forms a closed subspace of constant codimension. We can find a continuous family of isometries $\phi_s:\ker D_{u^s}E_{q_0}\to \mathcal{H}$ onto a fixed Hilbert space $\mathcal{H}$, such that 
    \begin{equation}
        \alpha(s)=1-\sup\{\langle \phi_s\circ Q_s \circ \phi^{-1}_s w,w\rangle_{\mathcal{H}} \mid \|w\|_{\mathcal{H}}=1, w\in \mathcal{H}\},
    \end{equation}
    where $\tilde{Q}_s=\phi_s\circ Q_s \circ \phi^{-1}_s$ is a continuous one-parameter family of symmetric and compact operators on $\mathcal{H}$. The supremum term is exactly the maximal eigenvalue of $\tilde{Q}_s$, which depends continuously on $s$ since $\tilde{Q}_s$ is continuous in the norm topology (see \cite[V Theorem 4.10]{kato}). Therefore, $\alpha(s)$ is continuous.

    For the first claim (1), by a standard rescaling, we have
    \begin{equation}
        D^2_{u^s}E_{q_0}(v,v)=\iint_{0\le \tau\le t\le 1}[(P^{u^s}_{\tau,1})_*X_{v(\tau)},(P^{u^s}_{t,1})_*X_{v(t)}]|_{\gamma_u(s)}d\tau dt,
    \end{equation}
    and $\lambda^s=s(P^u_{s,1})^*\lambda$. It directly follows that $Q_s\to 0$ when $s\to 0^+,$ hence $\tilde{Q}_s\to 0.$ Thus $\alpha(0) = 1$.

    Now let us show that the infimum in the definition of $\alpha(s)$ is attained, which serves as a foundation for properties (2) and (3). For a fixed $s$, choose a minimizing sequence $v_n \in \ker D_{u^s}E_{q_0}$ such that $\|v_n\|_s = 1$ and
    \begin{equation}
        \|v_n\|^2_s-\langle Q_s(v_n),v_n\rangle_s\to \alpha(s), \quad \text{as } n\to \infty.
    \end{equation}
    Up to the extraction of a subsequence, we can find a limit $\bar v$ such that $v_n\rightharpoonup \bar v$ weakly in $L^2.$ Since $D_{u^s}E_{q_0}$ is a continuous linear operator, the weak limit satisfies $\bar v\in \ker D_{u^s}E_{q_0}$. Moreover, by the lower semi-continuity of the norm with respect to weak convergence, we have $\|\bar v\|_s\le 1.$ By the compactness of $Q_s$, it maps weakly convergent sequences to strongly convergent sequences, yielding $\langle Q_s(v_n),v_n\rangle_s \to \langle Q_s(\bar v),\bar v\rangle_s$. Thus
    \begin{equation}
        \left\|\frac{\bar v}{\|\bar v\|_s}\right\|^2_s-\frac{1}{\|\bar v\|^2_s}\langle Q_s(\bar v),\bar v\rangle_s \le \alpha(s) + \left(\frac{1}{\|\bar v\|^2_s}-1\right)(\alpha(s)-1).
    \end{equation}
    Since $\alpha(s) \le 1$, if $\|\bar{v}\|_s < 1$, the right-hand side would be strictly less than $\alpha(s)$, which contradicts the definition of $\alpha(s)$ as an infimum. Hence, we must have $\|\bar v\|_s = 1$, meaning $\alpha(s)$ is attained at $\bar v$.

    To prove the second claim (2), if $\alpha(s)=0$, since the infimum is attained at some normalized vector $\bar v \neq 0$, 
    $I-Q_s$ is a positive semi-definite operator, which implies $(I-Q_s)\bar v = 0$. This means $I-Q_s$ is degenerate, and thus $\text{Hess}_{u^s}(J|_{E^{-1}_{q_0}(\gamma_u(s))})$ is degenerate.

    For the third claim (3), fixing $0 < s \le s'\le 1$ and $v\in\ker D_{u^s}E_{q_0}$ with $\|v\|_s=1$, we set 
    \begin{equation}
        \hat{v}(t):=\begin{cases}\sqrt{\frac{s'}{s}}v\left(\frac{s'}{s}t\right), & 0\le t\le \frac{s}{s'}\\
            0, & \frac{s}{s'}< t\le 1.
        \end{cases}
    \end{equation} 
    Then one can verify that $\|\hat{v}\|_{s'}=1$ and $\hat{v}\in \ker D_{u^{s'}}E_{q_0}$. Furthermore, a direct change of variables yields $\langle\lambda^s,D^2_{u^s}E_{q_0}(v,v)\rangle=\langle\lambda^{s'},D^2_{u^{s'}}E_{q_0}(\hat{v},\hat{v})\rangle.$ This implies $\alpha(s')$ is bounded above by the value at $\hat{v}$, which equals $\alpha(s)$, which conclude that $\alpha(s)$ is monotone decreasing on $(0,1]$.

    For the last claim (4), since $\alpha(s)$ is monotone decreasing, if $\alpha(\bar s)=0$, then $\alpha(s) \le 0$ for all $s > \bar s$. Suppose by contradiction that $\alpha(s) = 0$ for some $s > \bar s$. By monotonicity, $\alpha$ would be identically zero on the entire interval $[\bar s, s]$. This would imply that every point $\gamma_u(t)$ for $t \in [\bar s, s]$ is conjugate to $\gamma_u(0)$, which directly contradicts Proposition~\ref{prop:1} . Therefore, we must have $\alpha(s) < 0$ for all $s > \bar s$.
\end{proof}

Theorem~\ref{conjugatepoint} can be immediately concluded from the above discussion. We now give the proof.
\begin{proof}[Proof of Theorem~\ref{conjugatepoint}]
    We proceed by contradiction. Assume there exist $s,s'\in[0,1]$ such that $\gamma(s')$ 
    is conjugate to $\gamma(s)$ with $|s'-s|<1$. Restricting our attention to the segment $\gamma|_{[s,s']}$ and 
    applying (4) of Lemma~\ref{lm:1}, we deduce that the Hessian of the energy functional $J$ 
    possesses a negative eigenvalue. Therefore, we can construct a variation joining $\gamma(s)$ to $\gamma(s')$ with a strictly shorter length, 
    contradicting the assumption that $\gamma$ is a minimizing geodesic.
\end{proof}

\section{Positivity}
Here we show a crucial lemma in the proof of Theorem~\ref{jest}, which makes it possible 
to apply Minkowski's determinant theorem.

\begin{proposition}\label{propB1}
    Suppose $x,y\in M$ and there exists a unique geodesic connecting $x$ and $y$ that contains no abnormal segments, 
    given by 
    $\gamma(t)=\exp_x(td_x\phi)$, where $\phi$ is twice differentiable at $x$ and satisfies
    \begin{equation}
        \frac{1}{2}d^2_{SF}(x,y)=-\phi(x),\quad \text{and}\quad \frac{1}{2}d_{SF}^2(z,y)\ge-\phi(z),\quad \forall z\in M.
    \end{equation}
    Then for all $s\in (0,1)$ we have:
        \begin{itemize}
            \item[(a)] $\det(N^{\mathrm{V}}_0(t)^{-1})>0$ for all $t\in (0,1)$;
            \item[(b)] $N^{\mathrm{V}}_0(t)^{-1}N^{\mathrm{V}}_s(t)N^{\mathrm{V}}_s(0)^{-1}\ge 0$ for all $t\in(0,s]$;
            \item[(c)] $N^{\mathrm{V}}_0(t)^{-1}N(t)\ge 0$ for all $t\in(0,1)$.
        \end{itemize}
    Furthermore, if $\gamma(1)$ is not conjugate to $\gamma(0)$, the above properties hold for all $s\in(0,1]$
    and $t\in (0,1].$
\end{proposition}

Since the proof relies entirely on the Hamiltonian framework and is independent of the choice of directions, 
the procedure in the sub-Riemannian case can be directly adapted to the sub-Finslerian setting. Here we provide a sketch of 
the proof and highlight the main ideas.

\begin{proof}[Sketch of proof]
    We first prove (a). Since no point $\gamma(t)$ is conjugate to $\gamma(0)$ for $t\in(0,1)$, it suffices to 
    show that for small $t\in(0,1)$, $\det(N^{\mathrm{V}}_0(t)^{-1})>0$. To this end, by choosing a fixed Darboux moving frame $E_i,F_i$ ($i=1,\dots,n$) along the 
    normal extremal $\lambda(t)$ with $\lambda(0)=d_x\phi$, and setting $W(t)=N^{\mathrm{V}}_0(t)M^{\mathrm{V}}_0(t)^{-1}$, we obtain the Riccati equation:
    \begin{equation}
        \dot{W}=WRW+A^TW+WA+B,\quad W(0)=0.
    \end{equation}
    Applying the matrix Riccati comparison theorem (see \cite[Lemma A.4]{BR16}), we deduce that $W(t)\ge 0$ for small $t>0$.
    In addition, since $M^{\mathrm{V}}_0(0)=I_n$ and $N^{\mathrm{V}}_0(t)$ is non-degenerate for $t\in(0,1)$, we have $\det W(t)>0$
    for small $t>0$, yielding (a).

    To prove (b), by Lemma~\ref{jcblemma}, one can check that (b) is equivalent to showing that $S(t):=N^{\mathrm{V}}_0(t)^{-1}N^{\mathrm{H}}_0(t)$
    is non-increasing. This fact indeed holds due to the nonnegativity of Hamiltonian, and we refer to the proof in \cite[Lemma 74]{BR}, as the argument remains valid in our sub-Finslerian context.

    To prove (c), we note the following fact (see the proof in \cite[Claim 2.4]{RiemIneq}, whose approach can be carried out here):
    \begin{equation}\label{eq:B1}
        \frac{1}{2s}d^2_{SF}(z,\gamma(s))\ge \frac{1}{2}d^2_{SF}(z,y)-\frac{1-s}{2}d^2_{SF}(x,y), \quad \forall z\in M, s \in(0,1].
    \end{equation}

    Combining the inequality above with the condition that $\frac{1}{2}d^2_{SF}(z,y)+\phi(z)\ge 0$ for any $z\in M$, 
    with equality when $z=x$, we find that there exists a constant $C=C(s,x,y)$ such that 
    \begin{equation}
        \frac{1}{2s}d^2_{SF}(z,\gamma(s))+\phi(z)\ge C,\quad \forall z\in M, s\in(0,1),
    \end{equation}
    with equality when $z=x$.
    Therefore, the function $z\mapsto \frac{1}{2s}d^2_{SF}(z,\gamma(s))$ has a critical point at $z=x$, and its well-defined Hessian satisfies 
    \begin{equation}
        \mathrm{Hess}\left(\frac{1}{2s}d^2_{SF}(z,\gamma(s))+\phi(z)\right)\Big|_{z=x}\ge 0.
    \end{equation}
    Since $c_s(z):=\frac{1}{2s}d^2_{SF}(z,\gamma(s))$ is smooth in a neighborhood of $x$, we have 
    \begin{equation}
        \pi\circ e^{s\vec{H}}(d_z(-c_s))=\gamma(s),
    \end{equation}
    implying that $e^{s\vec{H}}_* (d^2_{x}(-c_s)) \subset \ker \pi_*$. Thus, the $n$-tuple 
    of Jacobi fields $e^{t\vec{H}}_*(-c_s)(X(0))$ satisfies  
    \begin{equation}\label{eq:B3}
        e^{t\vec{H}}_*\circ d^2_{x}(-c_s)(X(0))=E(t)\cdot M^v_s(t)N^v_s(0)^{-1}+F(t)\cdot N^v_s(t)N^v_s(0)^{-1},\quad t\in[0,1].
    \end{equation}
    On the other hand, via the definition of Jacobi matrix $\mathbf{J}(t)=\begin{pmatrix}
        M(t)\\ N(t)
    \end{pmatrix},$ we have
    \begin{equation}\label{eq:B4}
        e^{t\vec{H}}_*\circ d^2_{x}\phi(X(0))=E(t)\cdot M(t)+F(t)\cdot N(t),\quad t\in[0,1].
    \end{equation}
    Evaluating \eqref{eq:B3} and \eqref{eq:B4} at $t=0$, we have 
    \begin{equation}
        \begin{aligned}
            &d^2_x(-c_s)(X(0))=E(0)\cdot M^v_s(0)N^v_s(0)^{-1}+F(0)\cdot N^v_s(0)N^v_s(0)^{-1}\\
            &d^2_{x}\phi(X(0))=E(0)\cdot M(0)+F(0)\cdot N(0).
        \end{aligned}
    \end{equation}
    The non-negativity of the Hessian reads
    \begin{equation}
        0\le M(0)-M^v_s(0)N^v_s(0)^{-1}=M(0)+S(s).
    \end{equation}
    Via Lemma~\ref{jcblemma}, one can check that the above inequality actually implies (c), completing our proof.
\end{proof}

\section{Jacobian determinant of the Randers sub-Finslerian Heisenberg exponential map}

For the sake of completeness, we include the explicit computation concerning the Jacobian determinant of the Randers 
sub-Finslerian Heisenberg exponential map here, proving Lemma~\ref{dcr}. 
The derivation is essentially the same as that for the 
sub-Riemannian Heisenberg exponential map, the only modification being the replacement of $t\theta/2$ by $\frac{\psi-\psi_0}{2}$.

We consider the exponential map from the parameter space $(r, \theta, \psi_0)$ to the Heisenberg group. Let $\Delta \psi = \psi - \psi_0$, where $\psi$ is implicitly determined by the relation:
\begin{equation}
    \psi-\psi_0+a(\sin\psi-\sin\psi_0)=\theta t.
\end{equation}
For convenience, we introduce the following notations:
\begin{gather*}
    M=1+a\cos\psi, \quad L=1+a\cos\psi_0; \\
    \Delta S=\sin \psi-\sin\psi_0, \quad \Delta C=\cos\psi_0-\cos\psi, \quad K=1-\cos\Delta\psi.
\end{gather*}
Utilizing the implicit derivatives $\frac{\partial \psi}{\partial \theta}=\frac{t}{M}$ and $\frac{\partial\psi}{\partial\psi_0}=\frac{L}{M}$, the Jacobi matrix $\mathbf{J}=\frac{\partial (x,y,z)}{\partial (r,\theta,\psi_0)}$ is given by 
\begin{equation}
\mathbf{J}
=
\begin{pmatrix}
\dfrac{\Delta S}{\theta}
&
-\dfrac{r\Delta S}{\theta^2} + \dfrac{r t \cos\psi}{\theta M}
&
\dfrac{r(L\cos\psi - M\cos\psi_0)}{\theta M}
\\[1.2em]

\dfrac{\Delta C}{\theta}
&
-\dfrac{r\Delta C}{\theta^2} + \dfrac{r t \sin\psi}{\theta M}
&
\dfrac{r(L\sin\psi - M\sin\psi_0)}{\theta M}
\\[1.2em]

\dfrac{r(\Delta\psi-\sin\Delta\psi)}{\theta^2}
&
-\dfrac{r^2(\Delta\psi-\sin\Delta\psi)}{\theta^3} + \dfrac{r^2 t K}{2\theta^2 M}
&
\dfrac{r^2 K(L-M)}{2\theta^2 M}
\end{pmatrix}.
\end{equation}
Adding $\frac{r}{\theta}$ times the first column to the second column to eliminate the leading terms in the second column, we obtain
\begin{equation}
\mathbf{J}
=
\begin{pmatrix}
\dfrac{\Delta S}{\theta}
&
\dfrac{r t \cos\psi}{\theta M}
&
\dfrac{r(L\cos\psi - M\cos\psi_0)}{\theta M}
\\[1.2em]

\dfrac{\Delta C}{\theta}
&
\dfrac{r t \sin\psi}{\theta M}
&
\dfrac{r(L\sin\psi - M\sin\psi_0)}{\theta M}
\\[1.2em]

\dfrac{r(\Delta\psi-\sin\Delta\psi)}{\theta^2}
&
\dfrac{r^2 t K}{2\theta^2 M}
&
\dfrac{r^2 K(L-M)}{2\theta^2 M}
\end{pmatrix}.
\end{equation}
By extracting the common factor $\frac{1}{\theta}$ from the first column, $\frac{rt}{\theta M}$ from the second column, 
$\frac{r}{\theta M}$ from the third column, and finally pulling out $\frac{r}{\theta}$ from the third row, 
the determinant of the Jacobian becomes
\begin{equation}
    \det \mathbf{J} = \frac{r^3 t}{\theta^4M^2} \det A, 
\end{equation}
where the reduced matrix $A$ is defined as
\begin{equation}
    A :=
\begin{pmatrix}
\Delta S
&
\cos\psi
&
L\cos\psi - M\cos\psi_0
\\[1em]

\Delta C
&
\sin\psi
&
L\sin\psi - M\sin\psi_0
\\[1em]

\Delta\psi-\sin\Delta\psi
&
\dfrac{1}{2}K
&
\dfrac{1}{2}K(L-M)
\end{pmatrix}.
\end{equation}
To evaluate $\det A$, we expand along the third row, yielding
\begin{equation}\label{detA}
    \det A=(\Delta \psi-\sin\Delta \psi)C_{31} + \frac{1}{2}K C_{32} + \frac{1}{2}K(L-M)C_{33},
\end{equation}
where $C_{ij}$ represent the respective cofactors. Computing these minors explicitly, we obtain:
\begin{align*}
    C_{31} &= \det\begin{pmatrix}
        \cos\psi & L\cos\psi-M\cos\psi_0\\
        \sin\psi & L\sin\psi-M\sin\psi_0
    \end{pmatrix}
    = M\sin\Delta \psi; \\[1em]
    C_{32} &= -\det\begin{pmatrix}
        \Delta S &L\cos\psi-M\cos\psi_0\\
        \Delta C &L\sin\psi-M\sin\psi_0
    \end{pmatrix}
    = -K(L+M); \\[1em]
    C_{33} &= \det\begin{pmatrix}
        \Delta S &\cos\psi\\
        \Delta C&\sin\psi
    \end{pmatrix}
    = K.
\end{align*}
Substituting these minors back into \eqref{detA} gives
\begin{align}
    \det A &= (\Delta\psi - \sin\Delta\psi)(M\sin\Delta\psi) - \frac{1}{2}K^2(L+M) + \frac{1}{2}K^2(L-M) \notag \\
    &= M(\Delta\psi\sin\Delta\psi - \sin^2\Delta\psi) - MK^2 \notag \\
    &= M\big(\Delta\psi\sin\Delta\psi - (\sin^2\Delta\psi + K^2)\big).
\end{align}
Recalling that $K = 1 - \cos\Delta\psi$, we expand $K^2$ and utilize the identity $\sin^2\Delta\psi + \cos^2\Delta\psi = 1$, concluding that
\begin{equation*}
    \sin^2\Delta\psi + K^2 = \sin^2\Delta\psi + (1 - \cos\Delta\psi)^2 = 2 - 2\cos\Delta\psi.
\end{equation*}
By substituting this identity back, the expression further simplifies exactly to
\begin{equation}
    \det A = -M\big(2 - 2\cos\Delta \psi - \Delta\psi\sin\Delta\psi\big)=4M\sin\frac{\Delta\psi}{2}\big(\Delta \psi\cos\frac{\Delta\psi}{2}-\sin\frac{\Delta \psi}{2}\big).
\end{equation}
Combining the above computations, we obtain the formula for the Jacobian determinant, thereby proving Lemma~\ref{dcr}.

\medskip

\
\noindent{\sc Acknowledgment.}
The author is grateful to Professor Shin-ichi Ohta for his guidance, valuable comments, and support during this work, and to Kenshiro Tashiro for helpful discussions and suggestions. This work was supported by the China Scholarship Council under Grant No. 202506270114.

\phantomsection


\begin{thebibliography}{99}
    \bibitem{Agr}
Agrachev A, Barilari D, Boscain U.
A comprehensive introduction to sub-Riemannian geometry.
Cambridge University Press, 2019.

\bibitem{Ab}
Alabdulsada L M.
Sub-Randers metrics.
arXiv:2606.21922, 2026.

\bibitem{AK}
Alabdulsada L M, Kozma L.
Hopf-Rinow theorem of sub-Finslerian geometry.
arXiv:2301.13438, 2023.

\bibitem{Bac}
Bacher K.
On Borell-Brascamp-Lieb inequalities on metric measure spaces.
Potential Anal., 2010, 33(1): 1--15.

\bibitem{BCS}
Bao D, Chern S S, Shen Z.
An Introduction to Riemann-Finsler Geometry.
Springer, 2000.

\bibitem{BRS}
Bao D, Robles C, Shen Z.
Zermelo navigation on Riemannian manifolds.
J. Differential Geom., 2004, 66(3): 377--435.

\bibitem{BMR}
Barilari D, Mondino A, Rizzi L.
Unified synthetic Ricci curvature lower bounds for Riemannian and sub-Riemannian structures.
Mem. Amer. Math. Soc., 2026, 317(1613): viii+145 pp.

\bibitem{BR16}
Barilari D, Rizzi L.
Comparison theorems for conjugate points in sub-Riemannian geometry.
ESAIM Control Optim. Calc. Var., 2016, 22(2): 439--472.

\bibitem{BR}
Barilari D, Rizzi L.
Sub-Riemannian interpolation inequalities.
Invent. Math., 2019, 215(3): 977--1038.

\bibitem{BR2}
Barilari D, Rizzi L.
Bakry-Émery curvature and model spaces in sub-Riemannian geometry.
Math. Ann., 2020, 377(1): 435--482.

\bibitem{BMRT24}
Borza S, Magnabosco M, Rossi T, Tashiro K.
Measure contraction property and curvature-dimension condition on sub-Finsler Heisenberg groups.
arXiv:2402.14779, 2024.

\bibitem{BMRT1}
Borza S, Magnabosco M, Rossi T, Tashiro K.
Curvature exponent of sub-Finsler Heisenberg groups.
SIAM J. Math. Anal., 2025, 57(4): 3561--3586.

\bibitem{Borza}
Borza S, Rigoni C, Zoghlami O.
Hausdorff dimension and failure of synthetic curvature bounds in the sub-Lorentzian Heisenberg group.
arXiv:2509.06563, 2025.

\bibitem{BT23}
Borza S, Tashiro K.
Measure contraction property, curvature exponent and geodesic dimension of sub-Finsler $\ell^p$-Heisenberg groups.
arXiv:2305.16722, 2023.

\bibitem{CJT}
Chitour Y, Jean F, Trélat E.
Genericity results for singular curves.
J. Differential Geom., 2006, 73(1): 45--73.

\bibitem{CR}
Cannarsa P, Rifford L.
Semiconcavity results for optimal control problems admitting no singular minimizing controls.
Ann. Inst. H. Poincaré Anal. Non Linéaire, 2008, 25(4): 773--802.

\bibitem{RiemIneq}
Cordero-Erausquin D, McCann R J, Schmuckenschläger M.
A Riemannian interpolation inequality à la Borell, Brascamp and Lieb.
Invent. Math., 2001, 146(2): 219--257.

\bibitem{FF}
Fathi A, Figalli A.
Optimal transportation on non-compact manifolds.
Israel J. Math., 2010, 175(1): 1--59.

\bibitem{FR}
Figalli A, Rifford L.
Mass transportation on sub-Riemannian manifolds.
Geom. Funct. Anal., 2010, 20(1): 124--159.

\bibitem{Jea}
Jean F.
Control of nonholonomic systems: from sub-Riemannian geometry to motion planning.
Springer, 2014.

\bibitem{Jui2}
Juillet N.
Geometric inequalities and generalized Ricci bounds in the Heisenberg group.
Int. Math. Res. Not., 2009, 2009(13): 2347--2373.

\bibitem{Jui}
Juillet N.
Sub-Riemannian structures do not satisfy Riemannian Brunn-Minkowski inequalities.
Rev. Mat. Iberoam., 2021, 37(1): 177--188.

\bibitem{kato}
Kato T.
Perturbation theory for linear operators.
Springer, 1966.

\bibitem{LV}
Lott J, Villani C.
Ricci curvature for metric-measure spaces via optimal transport.
Ann. of Math., 2009, 169(3): 903--991.

\bibitem{MR}
Magnabosco M, Rossi T.
Failure of the curvature-dimension condition in sub-Finsler manifolds.
arXiv:2307.01820, 2023.

\bibitem{Mil}
Milman E.
The quasi curvature-dimension condition with applications to sub-Riemannian manifolds.
Comm. Pure Appl. Math., 2021, 74(12): 2628--2674.

\bibitem{MC}
Monti R, Serra Cassano F.
Surface measures in Carnot-Carathéodory spaces.
Calc. Var. Partial Differential Equations, 2001, 13(3): 339--376.

\bibitem{Oh0}
Ohta S.
Comparison Finsler Geometry.
Springer, 2021.

\bibitem{Oh2}
Ohta S.
On the measure contraction property of metric measure spaces.
Comment. Math. Helv., 2007, 82(4): 805--828.

\bibitem{Ohta}
Ohta S.
Finsler interpolation inequalities.
Calc. Var. Partial Differential Equations, 2009, 36(2): 211--249.

\bibitem{Ohta14}
Ohta S.
On the curvature and heat flow on Hamiltonian systems.
Anal. Geom. Metr. Spaces, 2014, 2(1): 81--114.

\bibitem{Rf0}
Rifford L.
Ricci curvature in Carnot groups.
Math. Control Relat. Fields, 2013, 3(4): 467--487.

\bibitem{Rf}
Rifford L.
Sub-Riemannian geometry and optimal transport.
Springer, 2014.

\bibitem{Rizzi2}
Rizzi L.
Measure contraction properties of Carnot groups.
Calc. Var. Partial Differential Equations, 2016, 55(3): 60.

\bibitem{RS23}
Rizzi L, Stefani G.
Failure of curvature-dimension conditions on sub-Riemannian manifolds via tangent isometries.
J. Funct. Anal., 2023, 285(9): 110099.

\bibitem{Sa}
Sarychev A V.
The index of the second variation of a control system.
Math. USSR Sb., 1982, 41(3): 383--401.

\bibitem{St2}
Sturm K-T.
On the geometry of metric measure spaces.
Acta Math., 2006, 196(1): 65--131.

\bibitem{St1}
Sturm K-T.
On the geometry of metric measure spaces II.
Acta Math., 2006, 196(2): 133--177.

\bibitem{Vi}
Villani C.
Optimal transport: old and new.
Springer, 2008.

\bibitem{YAD}
Yan Z, An H, Deng S.
The geometry of irreversible sub-Finslerian spaces.
Sci. China Math., 2025, 68(5): 1151--1176.

\bibitem{ZL}
Zelenko I, Li C.
Differential geometry of curves in Lagrange Grassmannians with given Young diagram.
Differential Geom. Appl., 2009, 27(6): 723--742.

\end{thebibliography}
\end{document}